\documentclass[10pt, a4paper, reqno]{amsart}

\usepackage[utf8]{inputenc}
\usepackage{a4wide}

\usepackage{amssymb, amsmath, amsthm}
\usepackage{mathtools}
\usepackage{thmtools}
\usepackage{enumitem}
\usepackage[dvipsnames]{xcolor}
\usepackage{ifthen}
\usepackage{tikz}
	\usetikzlibrary{positioning}
\usepackage[colorlinks=true, citecolor=Green, linkcolor=Red]{hyperref}
\usepackage{orcidlink}
\usepackage{bm}
\usepackage{stmaryrd}

\newcommand{\Fcal}{\mathcal{F}}

\newcommand{\Hcal}{\mathcal{H}}

\newcommand{\Pcal}{\mathcal{P}}

\newcommand{\Scal}{\mathcal{S}}

\newcommand{\Xbb}{\mathbb{X}}

\newcommand{\Nsf}{\mathsf{N}}

\newcommand{\Rsf}{\mathsf{R}}

\newcommand{\ta}{
  \mathchoice
    {\taaux{1.0}{true}{0pt}}
    {\taaux{1.0}{true}{0pt}}
    {\taaux{0.7}{false}{0.5ex}}
    {\taaux{0.5}{false}{0.5ex}}
}

\newcommand{\taaux}[3]{%
  \mathord{%
    \ifthenelse{\equal{#2}{true}}{\vcenter}{\raisebox{#3}}{%
      \hbox{%
        \scalebox{#1}{%
          \tikz[baseline=(X.base)]{
            \node (X) at (0,0) {};
            \draw[fill=black] (0,0) circle[radius=0.04];
            \node[font=\tiny, inner sep=0pt, outer sep=0pt] at (0,-0.18) {$\alpha$};
          }%
        }%
      }%
    }%
  }%
}

\newcommand{\tao}{
  \mathchoice
    {\taoaux{1.0}{true}{0pt}}
    {\taoaux{1.0}{true}{0pt}}
    {\taoaux{0.7}{false}{0.5ex}}
    {\taoaux{0.5}{false}{0.5ex}}
}

\newcommand{\taoaux}[3]{%
  \mathord{%
    \ifthenelse{\equal{#2}{true}}{\vcenter}{\raisebox{#3}}{%
      \hbox{%
        \scalebox{#1}{%
          \tikz[baseline=(X.base)]{
            \node (X) at (0,0) {};
            \draw[fill=black] (0,0) circle[radius=0.04];
            \node[font=\tiny, inner sep=0pt, outer sep=0pt] at (0,-0.18) {$ $};
          }%
        }%
      }%
    }%
  }%
}

\newcommand{\tazero}{
  \mathchoice
    {\tazeroaux{1.0}{true}{0pt}}
    {\tazeroaux{1.0}{true}{0pt}}
    {\tazeroaux{0.7}{false}{0.5ex}}
    {\tazeroaux{0.5}{false}{0.5ex}}
}

\newcommand{\tazeroaux}[3]{%
  \mathord{%
    \ifthenelse{\equal{#2}{true}}{\vcenter}{\raisebox{#3}}{%
      \hbox{%
        \scalebox{#1}{%
          \tikz[baseline=(X.base)]{
            \node (X) at (0,0) {};
            \draw[fill=black] (0,0) circle[radius=0.04];
            \node[font=\tiny, inner sep=0pt, outer sep=0pt] at (0,-0.18) {$\alpha_0$};
          }%
        }%
      }%
    }%
  }%
}

\newcommand{\tb}{%
  \mathchoice
    {\tbaux{1.0}{true}{0pt}}
    {\tbaux{1.0}{true}{0pt}}
    {\tbaux{0.7}{false}{0.5ex}}
    {\tbaux{0.5}{false}{0.5ex}}
}

\newcommand{\tbaux}[3]{%
  \mathord{%
    \ifthenelse{\equal{#2}{true}}{\vcenter}{\raisebox{#3}}{%
      \hbox{%
        \scalebox{#1}{%
          \tikz[baseline=(X.base)]{
            \node (X) at (0,0) {};
            \draw[fill=black] (0,0) circle[radius=0.04];
            \node[font=\tiny, inner sep=0pt, outer sep=0pt] at (0,-0.18) {$\beta$};
          }%
        }%
      }%
    }%
  }%
}

\newcommand{\tab}{%
  \mathchoice
    {\tabaux{1.0}{true}{0pt}}
    {\tabaux{1.0}{true}{0pt}}
    {\tabaux{0.7}{false}{0.5ex}}
    {\tabaux{0.5}{false}{0.5ex}}
}

\newcommand{\tabaux}[3]{%
  \mathord{%
    \ifthenelse{\equal{#2}{true}}{\vcenter}{\raisebox{#3}}{%
      \hbox{%
        \scalebox{#1}{%
          \tikz[baseline=(X.base)]{
            \node (X) at (0,0) {};
            \draw[fill=black] (0,0) circle[radius=0.04];
            \draw[fill=black] (0.20,0) circle[radius=0.04];
            \node[font=\tiny, inner sep=0pt, outer sep=0pt] at (0,-0.18) {$\alpha$};
            \node[font=\tiny, inner sep=0pt, outer sep=0pt] at (0.20,-0.18) {$\beta$};
          }%
        }%
      }%
    }%
  }%
}

\newcommand{\taboo}{%
  \mathchoice
    {\tabooaux{1.0}{true}{0pt}}
    {\tabooaux{1.0}{true}{0pt}}
    {\tabooaux{0.7}{false}{0.5ex}}
    {\tabooaux{0.5}{false}{0.5ex}}
}

\newcommand{\tabooaux}[3]{%
  \mathord{%
    \ifthenelse{\equal{#2}{true}}{\vcenter}{\raisebox{#3}}{%
      \hbox{%
        \scalebox{#1}{%
          \tikz[baseline=(X.base)]{
            \node (X) at (0,0) {};
            \draw[fill=black] (0,0) circle[radius=0.04];
            \draw[fill=black] (0.20,0) circle[radius=0.04];
            \node[font=\tiny, inner sep=0pt, outer sep=0pt] at (0,-0.18) {$ $};
            \node[font=\tiny, inner sep=0pt, outer sep=0pt] at (0.20,-0.18) {$ $};
          }%
        }%
      }%
    }%
  }%
}

\newcommand{\tba}{%
  \mathchoice
    {\tbaaux{1.0}{true}{0pt}}
    {\tbaaux{1.0}{true}{0pt}}
    {\tbaaux{0.7}{false}{0.5ex}}
    {\tbaaux{0.5}{false}{0.5ex}}
}

\newcommand{\tbaaux}[3]{%
  \mathord{%
    \ifthenelse{\equal{#2}{true}}{\vcenter}{\raisebox{#3}}{%
      \hbox{%
        \scalebox{#1}{%
          \tikz[baseline=(X.base)]{
            \node (X) at (0,0) {};
            \draw[fill=black] (0,0) circle[radius=0.04];
            \draw[fill=black] (0.20,0) circle[radius=0.04];
            \node[font=\tiny, inner sep=0pt, outer sep=0pt] at (0,-0.18) {$\beta$};
            \node[font=\tiny, inner sep=0pt, outer sep=0pt] at (0.20,-0.18) {$\alpha$};
          }%
        }%
      }%
    }%
  }%
}

\newcommand{\taazero}{%
  \mathchoice
    {\taazeroaux{1.0}{true}{0pt}}
    {\taazeroaux{1.0}{true}{0pt}}
    {\taazeroaux{0.7}{false}{0.5ex}}
    {\taazeroaux{0.5}{false}{0.5ex}}
}

\newcommand{\taazeroaux}[3]{%
  \mathord{%
    \ifthenelse{\equal{#2}{true}}{\vcenter}{\raisebox{#3}}{%
      \hbox{%
        \scalebox{#1}{%
          \tikz[baseline=(X.base)]{
            \node (X) at (0,0) {};
            \draw[fill=black] (0,0) circle[radius=0.04];
            \draw[fill=black] (0.20,0) circle[radius=0.04];
            \node[font=\tiny, inner sep=0pt, outer sep=0pt] at (0,-0.18) {$\alpha_0$};
            \node[font=\tiny, inner sep=0pt, outer sep=0pt] at (0.30,-0.18) {$\alpha_0$};
          }%
        }%
      }%
    }%
  }%
}

\newcommand{\tjd}{%
  \mathchoice
    {\tjdaux{1.0}{true}{0pt}}
    {\tjdaux{1.0}{true}{0pt}}
    {\tjdaux{0.7}{false}{0.5ex}}
    {\tjdaux{0.5}{false}{0.5ex}}
}

\newcommand{\tjdaux}[3]{%
  \mathord{%
    \ifthenelse{\equal{#2}{true}}{\vcenter}{\raisebox{#3}}{%
      \hbox{%
        \scalebox{#1}{%
          \tikz[baseline=(X.base)]{
            \node (X) at (0,0) {};
            \draw[fill=black] (0,0) circle[radius=0.04];
            \draw[fill=black] (0.20,0) circle[radius=0.04];
            \node[font=\tiny, inner sep=0pt, outer sep=0pt] at (0,-0.18) {$1$};
            \node[font=\tiny, inner sep=0pt, outer sep=0pt] at (0.20,-0.18) {$2$};
          }%
        }%
      }%
    }%
  }%
}

\newcommand{\tc}{%
  \mathchoice
    {\tcaux{1.0}{true}}
    {\tcaux{1.0}{true}}
    {\tcaux{0.7}{false}}
    {\tcaux{0.5}{false}}
}

\newcommand{\tcaux}[2]{%
  \mathord{%
    \ifthenelse{\equal{#2}{true}}{\vcenter}{ }{%
      \hbox{%
        \scalebox{#1}{%
          \tikz[baseline=(X.base)]{
            \node (X) at (0,0) {};
            \draw[fill=black] (0,0.10) circle[radius=0.04] -- (0,-0.10) circle[radius=0.04];
            \node[font=\tiny, inner sep=0pt, outer sep=0pt, right] at (0.08,0.10) {$\alpha$};
            \node[font=\tiny, inner sep=0pt, outer sep=0pt, right] at (0.08,-0.10) {$\beta$};
          }%
        }%
      }%
    }%
  }%
}

\newcommand{\td}{%
  \mathchoice
    {\tdaux{1.0}{true}}
    {\tdaux{1.0}{true}}
    {\tdaux{0.7}{false}}
    {\tdaux{0.5}{false}}
}

\newcommand{\tdaux}[2]{%
  \mathord{%
    \ifthenelse{\equal{#2}{true}}{\vcenter}{ }{%
      \hbox{%
        \scalebox{#1}{%
          \tikz[baseline=(X.base)]{
            \node (X) at (0,0) {};
            \draw[fill=black] (0,0.10) circle[radius=0.04] -- (0,-0.10) circle[radius=0.04];
            \node[font=\tiny, inner sep=0pt, outer sep=0pt, right] at (0.08,0.10) {$\beta$};
            \node[font=\tiny, inner sep=0pt, outer sep=0pt, right] at (0.08,-0.10) {$\alpha$};
          }%
        }%
      }%
    }%
  }%
}

\newcommand{\tcjd}{%
  \mathchoice
    {\tcjdaux{1.0}{true}}
    {\tcjdaux{1.0}{true}}
    {\tcjdaux{0.7}{false}}
    {\tcjdaux{0.5}{false}}
}

\newcommand{\tcjdaux}[2]{%
  \mathord{%
    \ifthenelse{\equal{#2}{true}}{\vcenter}{ }{%
      \hbox{%
        \scalebox{#1}{%
          \tikz[baseline=(X.base)]{
            \node (X) at (0,0) {};
            \draw[fill=black] (0,0.10) circle[radius=0.04] -- (0,-0.10) circle[radius=0.04];
            \node[font=\tiny, inner sep=0pt, outer sep=0pt, right] at (0.08,0.10) {$1$};
            \node[font=\tiny, inner sep=0pt, outer sep=0pt, right] at (0.08,-0.10) {$2$};
          }%
        }%
      }%
    }%
  }%
}

\newcommand{\tcoo}{%
  \mathchoice
    {\tcooaux{1.0}{true}}
    {\tcooaux{1.0}{true}}
    {\tcooaux{0.7}{false}}
    {\tcooaux{0.5}{false}}
}

\newcommand{\tcooaux}[2]{%
  \mathord{%
    \ifthenelse{\equal{#2}{true}}{\vcenter}{ }{%
      \hbox{%
        \scalebox{#1}{%
          \tikz[baseline=(X.base)]{
            \node (X) at (0,0) {};
            \draw[fill=black] (0,0.10) circle[radius=0.04] -- (0,-0.10) circle[radius=0.04];
            \node[font=\tiny, inner sep=0pt, outer sep=0pt, right] at (0.08,0.10) {$ $};
            \node[font=\tiny, inner sep=0pt, outer sep=0pt, right] at (0.08,-0.10) {$ $};
          }%
        }%
      }%
    }%
  }%
}

\newcommand{\tczero}{%
  \mathchoice
    {\tczeroaux{1.0}{true}}
    {\tczeroaux{1.0}{true}}
    {\tczeroaux{0.7}{false}}
    {\tczeroaux{0.5}{false}}
}

\newcommand{\tczeroaux}[2]{%
  \mathord{%
    \ifthenelse{\equal{#2}{true}}{\vcenter}{ }{%
      \hbox{%
        \scalebox{#1}{%
          \tikz[baseline=(X.base)]{
            \node (X) at (0,0) {};
            \draw[fill=black] (0,0.10) circle[radius=0.04] -- (0,-0.10) circle[radius=0.04];
            \node[font=\tiny, inner sep=0pt, outer sep=0pt, right] at (0.08,0.10) {$\alpha_0$};
            \node[font=\tiny, inner sep=0pt, outer sep=0pt, right] at (0.08,-0.10) {$\beta_0$};
          }%
        }%
      }%
    }%
  }%
}

\newcommand{\tvazero}{%
  \mathchoice
    {\tvazeroaux{1.0}{true}}
    {\tvazeroaux{1.0}{true}}
    {\tvazeroaux{0.7}{false}}
    {\tvazeroaux{0.5}{false}}
}

\newcommand{\tvazeroaux}[2]{%
  \mathord{%
    \ifthenelse{\equal{#2}{true}}{\vcenter}{ }{%
      \hbox{%
        \scalebox{#1}{%
          \tikz[baseline=(X.base)]{
            \node (X) at (0,0) {};
            \draw[fill=black] (0,0.10) circle[radius=0.04] -- (0,-0.10) circle[radius=0.04];
            \node[font=\tiny, inner sep=0pt, outer sep=0pt, right] at (0.08,0.10) {$\alpha_0$};
            \node[font=\tiny, inner sep=0pt, outer sep=0pt, right] at (0.08,-0.10) {$\alpha_0$};
          }%
        }%
      }%
    }%
  }%
}

\theoremstyle{plain}
\declaretheorem[name=Proposition, numberwithin=section]{proposition}
\declaretheorem[name=Theorem, sibling=proposition]{theorem}
\declaretheorem[name=Lemma, sibling=proposition]{lemma}

\theoremstyle{definition}
\declaretheorem[name=Remark, sibling=proposition]{remark}
\declaretheorem[name=Definition, sibling=proposition]{definition}

\declaretheorem[name=Example, sibling=proposition]{example}

\numberwithin{equation}{section}

\renewcommand{\d}[1]{\, \mathrm{d} #1}

\begin{document}

\title{Rough differential equations on manifolds via natural bundles}

\author{Ivan B\v{e}lohl\'{a}vek}
\address{
	Charles University, 
	Faculty of Mathematics and Physics,
	Sokolovsk\'{a} 49/83,
	186 75, Prague 8, Czech Republic
}
\email{ivan.belohlavek@pm.me}

\author{Petr \v{C}oupek\orcidlink{0000-0002-5360-6095}}
\address{
	Charles University, 
	Faculty of Mathematics and Physics,
	Sokolovsk\'{a} 49/83,
	186 75, Prague 8, Czech Republic
}
\email{coupek@karlin.mff.cuni.cz}

\thanks{P.\ \v{C}oupek was supported by the Czech Science Foundation, project no.\ 26-21423S}

\begin{abstract}
In the article, a novel framework for rough differential equations on finite-dimensio-nal smooth manifolds driven by branched rough paths is developed utilizing the theory of natural bundles. The role of vector fields is played by sections of certain associated fiber bundles. The solutions are defined in a generalized Davie sense via a local approximation in such a way that they are invariant under changes of coordinates. Existence and uniqueness of the solutions is proved and a necessary and sufficient condition for the invariance of a submanifold for the solution is given. The approach allows the treatment of rough differential equations driven by fully branched rough paths on manifolds without directly relying on a shuffle product formula, bracket extension, or the Connes--Kreimer Hopf algebra and without imposing additional structure on the manifold.
\end{abstract}

\keywords{Branched rough path, Rough differential equation, Manifold, Invariant manifold}

\subjclass[2020]{34F05, 60H10, 60L20}


\maketitle
\tableofcontents

\section{Introduction}
In the article, we develop a novel framework for the treatment of rough differential equations (RDEs) on smooth, finite-dimensional manifolds that are driven by $\Rsf^n$-valued level-2 branched rough paths via the theory of natural bundles. We prove existence and uniqueness of solutions and find necessary and sufficient conditions for the invariance of submanifolds.

\subsection{State of the art}
\subsubsection{Controlled differential equations}

Ordinary differential equations (ODEs) of the form 
	\begin{equation}
	\label{eq:ODE-intro}
		\d{y}^i = \sum_{\alpha=1}^n f_\alpha^i(y) \d{x}^\alpha, \quad i\in\{1,2,\ldots, m\},
	\end{equation}
where $f_\alpha^i :\Rsf^m\to \Rsf$ (for $i\in\{1,2,\ldots, m\}$ and $\alpha\in \{1,2,\ldots, n\}$) are smooth functions and $x^\alpha:[0,T]\to \Rsf$ (for $\alpha\in \{1,2,\ldots, n\}$) are continuous but not differentiable functions are ill-posed. 

The need for a rigorous mathematical treatment of such equations, besides general theoretical interest, arises from stochastic modeling---in this field, differential equations of the form~\eqref{eq:ODE-intro} describe the behavior of a dynamical system under uncertainty or its response to random forces. If, for example, we wish to model random non-systematic error, a natural choice for $x$ is a path of a \emph{Wiener process} \cite{KarShr98}. If, on the other hand, we wish to model errors with memory effects (such as the mean-reverting microstructure of financial markets in which overshoots tend to be quickly corrected or network traffic in which busy periods bring about more busy periods), a natural choice for $x$ could be a path of an \emph{$H$-fractional Brownian motion} (with $H<1/2$ in the former case and $H>1/2$ in the latter) \cite{BiaHuOksZha08} or a path of another non-Gaussian fractional process \cite{BaiTaq14} if we wish to account for asymmetry of the data.

\subsubsection{RDEs in Euclidean space}
One framework that allows for a rigorous treatment of equations of the form~\eqref{eq:ODE-intro} is the \emph{theory of rough paths} initiated in \cite{Lyo98}. The theory is based on the observation that, in the smooth case, the solution has a Taylor approximation
	\begin{equation*}
		y^i(t)-y^i(s) 
			\approx
				\sum_{\alpha=1}^n f^i_\alpha(y(s))\int_{s}^t \d{x}^\alpha(u) 
				+ \sum_{j=1}^m \sum_{\alpha,\beta=1}^n \left(\frac{ \partial f^i_\beta}{\partial y^j}f^j_\alpha\right)(y(s))\int_{s}^t\int_s^{u_1}\d{x}^\alpha(u_2)\d{x}^\beta(u_1) + \cdots
	\end{equation*}
up to any order. This suggests that, in general, the solution depends on the (iterated) integrals
	\[
		\int_{s}^t \d{x}^\alpha(u), 
			\quad 
		\int_{s}^t\int_s^{u_1}\d{x}^\alpha(u_2)\d{x}^\beta(u_1), 
		 	\quad 
		 \cdots
	\]
By seeking a \emph{classical} solution, we implicitly interpret the integrals as Riemann--Stieltjes integrals which are, if the path $x$ is smooth, well-defined and determined by $\int_s^t \d{x}^\alpha$. If, however, the path $x$ is less regular, then the higher-order iterated integrals cannot be understood as Riemann--Stietljes integrals and they must be supplied as additional input, the so-called \emph{rough path}.

Generally, a rough path is a two-variate map that takes values in the character group of a graded Hopf algebra (which encodes shuffle relations between iterated integrals) that satisfies Chen's identity (which encodes time multiplicativity) and analytic regularity conditions (e.g., H\"older \cite{FriHai20}, $p$-variation \cite{FriVic10}, Sobolev \cite{LiuProTei21}, Besov \cite{FriSee22}, or Besov--Orlicz \cite{CouHenSla26}). 

For example, for the so-called \emph{geometric} rough paths, the underlying algebra is the truncated tensor algebra which encodes the shuffle rules that are satisfied by iterated integrals of smooth paths. A prototypical example of a geometric rough path is the Stratonovich lift of a Wiener path; see \cite{FriVic10}. Another example are the \emph{branched} rough paths whose underlying algebra is the Connes--Kraimer Hopf algebra of rooted trees which arises when the shuffle relations are dropped. An example of a branched rough path is the It\^o lift of a Wiener path; \cite{Gub04}.

There are several notions of RDE solutions; see, e.g., \cite{Bai15, Gub04, Lyo98}. A particularly transparent solution concept is that of Davie \cite{Dav08} (see also \cite{FriVic10}): For a given lift of a path $x$ to a geometric rough path $(X,\Xbb)$ of finite $p$-variation with $p\in [2,3)$, a solution to equation~\eqref{eq:ODE-intro} is defined as the unique path satisfying the local Taylor expansion 
	\begin{equation}
	\label{eq:Taylor-intro}
		y^i(t)-y^i(s) 
			\approx 
				\sum_{\alpha=1}^n f_\alpha^i(y(s))X^\alpha(s,t)
				+ \sum_{j=1}^m \sum_{\alpha,\beta=1}^n \left(\frac{ \partial f_\beta^i}{\partial y^j}f^j_\alpha\right)(y(s))\Xbb^{\alpha,\beta}(s,t).
	\end{equation}
Such a definition is quite natural as it does not involve any a priori defined form of abstract integration and the solution can be constructed as the limit of an Euler-type discrete approximation scheme. Moreover, Davie solutions can be generalized to RDEs driven by branched rough paths \cite{Bai21, CasWei17, Gub10}.

\subsubsection{RDEs on manifolds}

In the past decade, there has been significant effort to extend the rough path theory to the manifold setting. Such attempts, however, typically only focus on geometric rough paths. We list the main approaches in the geometric setting.

In \cite{CasLitLyo11}, the notion of \emph{Lipschitz-$\gamma$ manifolds} is defined. In there, the transition functions between coordinate patches are Lipschitz-$\gamma$ in the sense of \cite{Ste70} with a global constant. Such setting allows one to define a $p$-rough path on a manifold intrinsically as a linear functional on sufficiently regular one-forms satisfying certain continuity and consistency properties. The RDE is then transferred onto a local coordinate chart, identified with a Euclidean RDE there, solved by classical RDE arguments, and the solution is then transferred back onto the manifold. The $\gamma$-Lipschitzianity of the transition maps allows one to obtain uniform estimates on the norms of the transformed vector fields across different charts so that the Picard iteration estimates are stable under changes of coordinates. This ensures that local solutions can be glued into a global solution. As such, the Lipschitz-$\gamma$ property of the manifold is essential to the approach but presents an additional structure that is intimately linked to the regularity of the driving path ($\gamma>p$). It is not clear which Lipschitz-$\gamma$ atlas should be chosen given that a smooth manifold is compatible with many distinct ones and neither it is clear how this choice effects the resulting objects. The work is extended in \cite{BouLyo22}.

A different, extrinsic, view is taken in \cite{CasDriLit15}. In there, the authors define a class of the so-called \emph{constrained rough paths} that consists of rough paths in the ambient Euclidean space whose first level lies on an embedded smooth submanifold and higher levels are constrained so that their tangential components lie in the tensor algebra generated by the tangent spaces along the path with any normal components being determined by the curvature of the submanifold. Solutions to RDEs are then intrinsically defined via a Davie-type approximation adapted to constrained rough paths and their existence and uniqueness is proved by extending the vector fields to the ambient Euclidean space, solving the thus obtained Euclidean RDE by classical arguments, and showing that the solution is a constrained rough path in the embedded submanifold. Integration of $1$-forms is developed to describe the solutions intrinsically using such integrals.

Yet another approach is taken in \cite{DriSem17}. In that article, two additional structures on the manifold---the so-called \emph{logarithm} and \emph{parallelism}---are introduced to give meaning to the difference of two points. This allows to define the notions of a controlled rough path, controlled rough $1$-form along such paths, and that of a rough integral of a (controlled rough) $1$-form and, subsequently, consider RDEs on manifolds. The solutions to such RDEs are again defined in a Davie sense (but shown to be equivalent to other definitions) and its existence and uniqueness are again proved by expressing the equations in local coordinates and solving thus obtained RDEs in the Euclidean space by classical arguments.

There are two works that develop a theory of RDEs on manifolds driven by branched rough paths of which we are aware: \cite{ArmBriCasFer22} (and its extension in \cite{Fer25}) and \cite{CurEbrManMun20}. 

In \cite{ArmBriCasFer22}, a branched rough path is represented locally in charts and by choosing a connection on the manifold, the authors are able to define a rough integral of controlled paths in a coordinate-invariant manner. Solutions to RDEs are defined in such a way that their coordinate expressions satisfy associated RDEs in the Euclidean space that explicitly depend on the correction terms coming from the non-geometric nature of the driving rough path and the chosen connection. In \cite{Fer25}, it is further clarified that the theory essentially relies on the explicit choice of the \emph{bracket extension} of a branched rough path \cite{Kel12} which specifies the form of its change-of-variable formula. Because of the complexity of this notion, only RDEs driven by \emph{quasi-geometric rough paths} \cite{Bel23} are considered.

In \cite{CurEbrManMun20}, the so-called \emph{planarly branched rough paths} are considered. Planarly branched rough paths are branched rough paths indexed only by planar rooted trees, i.e.\ by rooted trees with a specified ordering of the children of each vertex. This is realized by replacing the Connes--Kraimer Hopf algebra used for a general branched rough path by the non-commutative Munthe--Kaas--Wright Hopf algebra \cite{MunWri08}. This extra structure allows to consider and solve RDEs in \emph{homogeneous} spaces, i.e.\ in manifolds equipped with a transitive action of a finite-dimensional Lie group, because the planar structure encodes the order in which infinitesimal group actions are composed in the Davie-type expansion defining the solution.

\subsubsection{Invariance of submanifolds for RDEs}

An important problem is the characterization of invariant submanifolds for the solutions to RDEs. Let $f_\alpha$ (for $\alpha\in \{1,2,\ldots,n\}$) be smooth, compactly supported vector fields on a finite-dimensional smooth manifold $M$ and let $S$ be a properly embedded submanifold of manifold $M$. It is well-known that if at all points $p\in S$, the vector fields $f_\alpha$ are tangent to $S$, then for all smooth drivers $x^\alpha: [0,T]\to\Rsf$ (for $\alpha\in\{1,2,\ldots, n\})$, the solution $y$ to the ODE
	\[ 
		\d{y} = \sum_{\alpha=1}^n f_\alpha(y)\d{x}^\alpha
	\]
started at a point $p\in S$, will stay in $S$; see, e.g., \cite{Lee12}. It is an important question whether, and in which form, this result transfers to RDEs. But while there are many results concerning invariance of sets for classical ODEs and stochastic differential equations (SDEs), there are only a few works dealing with some form of invariance of sets for RDEs.

In \cite{CouMar17}, the viability conditions for a compact or convex set in the spirit of \cite{Nag42} (for ODEs) and \cite{AubDaP90} (for SDEs) for an RDE in $\Rsf^m$ driven by $\Rsf^n$-valued geometric rough path of arbitrary regularity are given.

In \cite{NeaKue21} and \cite{VarRie25}, the long-time behavior of solutions to RDEs in $\Rsf^m$ driven by $\Rsf^n$-valued rough path from the perspective of dynamical systems is studied and sufficient conditions for the existence of a local stable, unstable, and center manifolds for the induced flows and given. In \cite{NeaKue21}, the rough path is not necessarily geometric but it is of level-$2$ while in \cite{VarRie25} only geometric rough paths are considered but it is of level-$3$.

The work most closely related to our approach is the article \cite{Tap24}. In there, the author characterizes the invariance of finite-dimensional $C^3$-submanifolds for an RDE in a Banach space in the sense above. In his setting, the RDE is driven by a level-$2$ rough path $X$ in a Banach space $V$ and its bracket is assumed to be of the form $[X]_t=xt$ for a fixed symmetric element $x\in V\otimes V$.

\subsection{Current approach and description of results}
In the present work, we develop an alternative approach to RDEs on a smooth $m$-dimensional manifold $M$ driven by $\Rsf^n$-valued level-2 branched rough paths that is based on generalized Davie-type local expansions of the solutions. Subsequently, we use the obtained results to study invariance of submanifolds for such RDEs.

Within our framework, $\Rsf^n$-valued level-2 branched rough paths are simply tuples $(X^\tau)$ of two-variate functions indexed as in \cite{Fer25} by forests $\tau$ of the form $\ta$, $\tab$, or $\tc$ that satisfy the usual algebraic and analytic conditions; see \autoref{def:rough_path}. In order to describe RDEs, we define the notion of \emph{rough velocity fields} on $M$. Rough velocity fields on $M$ are sections of a certain associated fiber bundle; see \autoref{def:rvf}. They are constructed in such a way that they encode the response of the modeled system to the infinitesimal change in each of the components of the driving rough path and their local coordinate representations---given, in a chart $(U,x)$ on $M$, by a collection $(f_\tau)$ of smooth functions $f_\tau: U\to \Rsf^m$---transform naturally under changes of coordinates. As suggested by expansion~\eqref{eq:Taylor-intro}, for a rough velocity field to describe an RDE, it has to satisfy a compatibility condition, which we call \emph{admissibility}, between the first- and second-order data. Precisely, we require that
	\[
		\sum_{j=1}^m \frac{\partial f_{\tb}}{\partial y^j} f^j_{\ta} = f_{\tab} + f_{\tc}
	\]
holds in every chart $(U,x)$ on $M$, see \autoref{def:admissible-rvf}. Intuitively, this means that the first-order response $f_{\ta}$ determines the overall second-order response 
	$
		\sum_{j=1}^m \frac{\partial f_{\tb}}{\partial y^j} f^j_{\ta}
	$
which is distributed among the second-order components $f_{\tab}$ and $f_{\tc}$. Finally, given an $\Rsf^n$-valued level-$2$ branched rough path $X=(X^\tau)$ and an admissible rough velocity field $f$ on $M$, we define a solution to the RDE 
	\begin{equation}
	\label{eq:RDE-intro}
		\d{Y} = f(Y)\d{X}
	\end{equation}
as an $M$-valued function $Y$ that can be locally approximated in any chart $(U,x)$ as 
	\[ 
		Y^i(t)-Y^i(s) \approx \sum_{\alpha=1}^n f^i_{\ta}(Y(s))X^{\ta}(s,t) + \frac{1}{2} \sum_{\alpha,\beta=1}^n f^i_{\tab}(Y(s)) X^{\tab}(s,t) + \sum_{\alpha,\beta=1}^n f^i_{\tc}(Y(s)) X^{\tc}(s,t).
	\]
The construction of admissible rough velocity fields ensures that these local expansions are compatible under changes of coordinates and define a unique curve that we call the \emph{integral curve of $f$ driven by $X$}; see \autoref{def:integral-curve}. Initially, we assume that $M=\Rsf^m$ and that the rough velocity field $f$ is compactly supported. We argue why the existence of a global solution cannot be proved by Picard's iterations and use Euler's approximation as in \cite{Dav08} instead; see \autoref{thm:R-existence}. Uniqueness is also proved; see \autoref{thm:R-uniqueness}. On a general manifold $M$, we start by a uniqueness result; see \autoref{thm:M-uniqueness}, and then prove local existence of the solution by transferring the global existence result from $\Rsf^m$ to $M$; see \autoref{thm:M-local-existence}. If $f$ is compactly supported, we obtain global existence of the solution by gluing solutions constructed in overlapping charts; see \autoref{thm:M-global-existence}. 

Subsequently, we apply this framework to study the invariance of an embedded manifold for the RDE~\eqref{eq:RDE-intro}. We obtain necessary and sufficient conditions on the rough velocity field under which the solution started on a submanifold does not leave this submanifold; see \autoref{thm:invariance}.

Our approach allows for the treatment of RDEs on finite-dimensional smooth manifolds driven by, possibly fully branched, level-2 rough paths without relying on a shuffle product formula, bracket extension, or the Connes--Kraimer Hopf algebra; and without imposing additional structure on the manifold. The necessary structure is encoded in the rough velocity field instead. The obtained characterization of the invariance of submanifolds for RDEs describes the invariance in full generality and it is not limited to the special situation of \cite{Tap24} in which the RDE is Euclidean and the driving rough path has a specifically fixed bracket.

\subsection{Organization of the article} The article is organized as follows. In \autoref{sec:brp-davie}, we define level-2 branched rough paths and study transformation laws of Davie-type expansions. In \autoref{sec:rvf}, we recall basic results on jets and define rough velocities, the composition of jets and rough velocities, and construct the rough velocity bundle. Finally, we define rough velocity fields as smooth sections of the rough velocity bundle. \autoref{sec:RDEs} is devoted to RDEs on manifolds. In \autoref{sec:integral-curves}, we define integral curves and study their properties. In \autoref{sec:admissible-rvf}, we argue that not all rough velocity fields can be expected to have integral curves for all drives and arrive at the notion of admissible rough velocity fields. In \autoref{sec:R-existence} and \autoref{sec:R-uniqueness}, we treat the existence and uniqueness of integral curves in $\Rsf^m$, respectively, while in \autoref{sec:M-existence-uniqueness}, existence an uniqueness of integral curves on manifolds are treated. Finally, \autoref{sec:invariance} is devoted to the invariance of submanifolds for RDEs. 

\subsection{Notation and conventions} The following notation and conventions are used throughout the article. We use the superscript to indicate the coordinate; e.g.\ $x^i$ is the $i$\textsuperscript{th} coordinate of $x$. We also use the summation convention that two consecutive indexes in an expression are summed over; e.g., the expression $a_jx^j$ stands for $\sum_j a_jx^j$. On $\Rsf^m$, we use the norm $|x|:=\max_{i\in\{1,2,\ldots, m\}} |x^i|$ for simplicity. Whenever we talk about smooth objects, we always mean $C^\infty$. (It is obvious that this assumption can be weakened in most cases.) For integers $i\leq j$, we will use $\llbracket i,j \rrbracket$ to denote $\{i,i+1,\ldots, j-1,j\}$. We will also denote $\Delta I:= \{(s,t)\in I^2: s\leq t\}$ for a set $I\subseteq\Rsf$. Finally, we use the symbol $A\lesssim B$ to indicate that there exists a finite positive constant $C$ such that $A\leq CB$.

\section{Branched rough paths and Davie solutions}
\label{sec:brp-davie}

In this section, we define the notion of a branched rough path and recall the notion of Davie solutions to RDEs. We show that Davie solutions are invariant under change of coordinates only if the driving rough path is geometric. Then we show how Davie solutions should be extended so that the new solution concept is invariant under change of coordinates even for branched rough paths. Fix an interval $[S,T]\subseteq\Rsf$ throughout this section.

\subsection{Branched rough paths} 
Recall that a function $\omega: \Delta[S,T]\to [0,\infty)$ is called a \emph{control function on the interval $[S,T]$} if it is continuous, zero on the diagonal (i.e.\ $\omega(s,s) = 0$ holds for all $S\leq s\leq T$), and superadditive (i.e.\ $\omega(s,t) \geq \omega(s,u) + \omega(u,t)$ holds for all $S\leq s\leq u\leq t\leq T$). We will use the following observation: If $(s,t)\in\Delta[S,T]$ satisfies $\omega(s,t)\leq 1$, then
	\begin{equation}
	\label{eq:control}
			\omega(x,y)^{\gamma_1} \leq \omega(s,t)^{\gamma_2}, \quad (x,y)\in\Delta[s,t], \quad \gamma_1\geq \gamma_2.
	\end{equation}

For $n\in\Nsf$, denote by $\Fcal_n^2$ the set of forests of rooted trees labeled by $\llbracket 1,n \rrbracket $ with at most two vertices but excluding the empty forest. That is, $\Fcal_n^2$ consists of $n$ one-vertex trees of the form $\ta$, $n(n+1)/2$ two-vertex forests of the form $\tab$, and $n^2$ two-vertex trees of the form $\tc$, where $\alpha$ and $\beta$ are labels from the set $\llbracket 1,n \rrbracket$. A rough path is defined now.

\begin{definition}
\label{def:rough_path}
Let $n\in\Nsf$, $p\in [1,3)$, and let $\omega$ be a control function on $[S,T]$. An \emph{$n$-dimensional branched rough path on $[S,T]$ with regularity $p$ and control $\omega$} is a tuple $X=(X^\tau)$ consisting of functions $X^\tau: \Delta[S,T]\to\Rsf$ (for $\tau\in\Fcal_n^2$) such that the equalities
	\begin{align}
		X^{\ta}(s,t) &= X^{\ta}(s,u) + X^{\ta}(u,t),
		\label{eq:rp-additive}
		\\
		X^{\tab}(s,t) & = X^{\ta}(s,t)X^{\tb}(s,t)
		\label{eq:rp-multiplicative}
		\\
		X^{\tc}(s,t) & = X^{\tc}(s,u) + X^{\tc}(u,t) + X^{\ta}(s,u)X^{\tb}(u,t),
		\label{eq:rp-chen}
	\end{align}
hold for for every $\alpha,\beta\in\llbracket 1,n \rrbracket$ and every $(s,u,t)\in [S,T]^3$ satisfying $s\leq u\leq t$ and such that there exists a constant $C\in (0,\infty)$ such that the inequalities
	\begin{equation}
	\label{eq:rp_regularity}
		\left|X^{\ta}(s,t)\right| \leq C \omega(s,t)^\frac{1}{p},
		\quad 
		\left|X^{\tab}(s,t)\right| \leq C \omega(s,t)^\frac{2}{p}, 
		\quad\mbox{and}\quad 
		\left|X^{\tc}(s,t)\right| \leq C \omega(s,t)^\frac{2}{p}
	\end{equation}
hold for for all $\alpha,\beta\in\llbracket 1,n \rrbracket$ and $(s,t)\in\Delta[S,T]$. The least such $C$ is denoted by the symbol $\|X\|$.
\end{definition}

\begin{example}
\label{ex:natural-lift}
Let $n\in\Nsf$ and $p\in [1,3)$. Denote the space of continuous functions $x:[S,T]\to\Rsf^n$ of finite $p$-variation by $C^{p-\mathrm{var}}([S,T], \Rsf^n)$, see \cite[Definition 5.1]{FriVic10}, and let $x\in C^{p-\mathrm{var}}([S,T], \Rsf^n)$. If $p\in [1,2)$, then $x$ can be lifted to an $n$-dimensional branched rough path $X=(X^\tau)$ of regularity $p$ and control $\omega$ defined by $\omega(s,t):= \|x\|_{p-\mathrm{var};[s,t]}^p$ (for $(s,t)\in\Delta[S,T]$) by setting 
	\begin{align*}
		X^{\ta}(s,t)  := x^{\alpha}_t-x^{\alpha}_s,
			\quad 
		X^{\tab}(s,t) := (x^{\alpha}_t-x^{\alpha}_s)(x^{\beta}_t-x^{\beta}_s), 
			\quad\mbox{and}\quad
		X^{\tc}(s,t)  := \int_s^t (x^\alpha_r-x^\alpha_s) \d{x}_r^\beta
	\end{align*}
for $\alpha,\beta\in \llbracket 1,n \rrbracket $ and $(s,t)\in \Delta[S,T]$. The integral in the last expression is the classical Young integral, see, e.g., \cite[Chapter 6]{FriVic10}. If $p\in [2,3)$, then one cannot use the Young integral and the last object $X^{\tc}$ that satisfies Chen's relation~\eqref{eq:rp-chen} and the third regularity condition in~\eqref{eq:rp_regularity} has to be supplied as additional input. 
\end{example}

\begin{remark}
In most definitions of branched rough paths found in the literature, some variant of the Connes--Kreimer Hopf algebra $\mathcal{H}_{CK}$ is used. The underlying vector space of this algebra is the free ($\Rsf$-)vector space generated by $\Fcal_{n}$, the set of forests labeled by $\llbracket 1,n \rrbracket$. For example, in the article~\cite{Gub10}, which introduced branched rough paths, they are defined as algebra morphisms from $\Hcal_{CK}$ to an algebra of two-parameter paths. In~\cite{TapZam20}, they are defined as functions from $[0, 1]$ to the character group of $\Hcal_{CK}$---the group of multiplicative linear functionals on $\Hcal_{CK}$. In~\cite{Fer25}, they are defined as functions from an interval $[S,T]$ to the truncated Grossman--Larson Hopf algebra. This is again an algebra whose underlying vector space is generated by a set of labeled forests. 
In any case, a branched rough path may be seen as having components indexed by labeled forests which satisfy~\eqref{eq:rp-additive}--\eqref{eq:rp-chen}. We do not require the algebraic framework
but our \autoref{def:rough_path} encompasses the same idea.
\end{remark}

\subsection{Davie solutions and change of coordinates}
We begin by recalling a solution concept to RDEs in $\Rsf^m$ which captures the core of Davie's original definition; cf.\ \cite[Definition 3.1]{Dav08}.

\begin{definition}
\label{def:davie-solution}
Let $m,n\in\Nsf$, $p\in [1,3)$, $x\in C^{p-\mathrm{var}}([S,T],\Rsf^n)$, and let $f_\alpha^i:\Rsf^m\to\Rsf$ (for $i\in\llbracket 1, m \rrbracket$ and $\alpha\in\llbracket 1,n \rrbracket$) be smooth functions. Let also $Y_S\in\Rsf^m$. Consider the system of controlled differential equations
	\begin{equation}
	\label{eq:ode}
		\d{Y}^i  = f_{\alpha}^i(Y)\d{x}^\alpha, 
			\quad 
		Y^i(S) = Y_S^i,
	\end{equation}
for $i\in \llbracket 1, m \rrbracket$. Given a rough path lift $(X^\tau)$ of function $x$ (see \autoref{ex:natural-lift}), we say that a function $Y: [S,T]\to\Rsf^m$ is a \emph{Davie solution} to problem~\eqref{eq:ode} if $Y(S) = Y_S$ and if there exists a constant $C\in (0,\infty)$ such that the inequality
	\begin{equation}
	\label{eq:davie}
		\left|
			Y^i(t) - Y^i(s) 
			- f_\alpha^i(Y(s)) X^{\ta}(s,t) 
			- \left(
				\frac{\partial f_\beta^i}{\partial y^j}f_\alpha^j
			  \right)
			(Y(s)) X^{\tc} (s,t)
		\right| 
		\leq C\omega(s,t)^\frac{3}{p}
	\end{equation}
holds for every $i\in \llbracket 1,m \rrbracket$ and $(s,t)\in \Delta [S,T]$.
\end{definition}

Let us now investigate how expansions of the form~\eqref{eq:davie} behave under compositions with smooth functions. To this end, we have the following

\begin{lemma}
\label{thm:ch-var}
Let $m,n,\ell\in\Nsf$, let $U\subseteq\Rsf^m$ be an open convex set, and let $\phi:U\to\Rsf^\ell$ be a smooth function. Let $p\in [1,3)$, let $\omega$ be a control function on interval $[S,T]$, and let $X=(X^\tau)$ be an $n$-dimensional branched rough path on interval $[S,T]$ with regularity $p$ and control $\omega$. Let $Y: [S,T]\to U$ be a continuous curve in $U$ and let $f_\tau^i:U\to\Rsf$ (for $\tau\in\Fcal_n^2$ and $i\in\llbracket 1,m \rrbracket$) be smooth functions. Assume that there exists $C\in (0,\infty)$ such that the inequality 
	\begin{equation}
	\label{eq:full-davie}
		\bigg| 
			Y^i(t) - Y^i(s) 
			- f_{\ta}^i(Y(s)) X^{\ta}(s,t) 
			- \frac{1}{2} f_{\tab}^i(Y(s)) X^{\tab}(s,t) 
			- f_{\tc}^i(Y(s)) X^{\tc}(s,t) 
		\bigg| 
		\leq C \omega(s,t)^\frac{3}{p}
	\end{equation}
holds for every $i\in\llbracket 1,m \rrbracket$ and $(s,t)\in\Delta[S,T]$. Then there is a constant $\widetilde{C}\in (0,\infty)$ such that 
	\[
		\bigg|
			\phi^i(Y(t)) - \phi^i(Y(s)) 
			- g^i_{\ta}(Y(s))X^{\ta}(s,t) 
			- \frac{1}{2} g^i_{\tab}(Y(s)) X^{\tab}(s,t) 
			- g^i_{\tc}(Y(s))X^{\tc}(s,t) 
		\bigg| 
		\leq \widetilde{C}\omega(s,t)^\frac{3}{p},
	\]
where 
	\[
		g^i_{\ta} := \frac{\partial\phi^i}{\partial y^j}f_{\ta}^j, 
			\quad 
		g_{\tab}^i := \frac{\partial\phi^i}{\partial y^j}f_{\tab}^j + \frac{\partial^2 \phi^i}{\partial y^j\partial y^k} f_{\ta}^jf_{\tb}^k, 
			\quad\mbox{and}\quad  
		g_{\tc}^i := \frac{\partial\phi^i}{\partial y^j}f_{\tc}^j,
	\]
holds for every $i\in \llbracket 1,\ell \rrbracket$ and $(s,t)\in\Delta[S,T]$.
\end{lemma}

\begin{proof}
Let $i\in \llbracket 1,\ell \rrbracket$ and let $(s,t)\in\Delta[S,T]$ satisfy $\omega(s,t)\leq 1$. Taylor's expansion of the function $\phi$ at $Y(s)$ yields
	\begin{multline*}
		\Bigg| 
			\phi^i(Y(s)) - \phi^i(Y(s)) 
			+ \frac{\partial\phi^i}{\partial y^j}(Y(s))(Y^j(t) - Y^j(s)) \\
			+ \frac{1}{2} \frac{\partial^2 \phi^i}{\partial y^j\partial y^k} (Y(s)) (Y^j(t)-Y^j(s))(Y^k(t)-Y^k(s))
		\Bigg| \lesssim |Y(t)-Y(s)|^3,
	\end{multline*}
where the constant can be taken independent of $s$, as it only depends on the derivatives of $\phi$ at a point in the convex hull $K$ of the trajectory of the curve $Y$ that is a compact subset of $U$. Now, by using inequality~\eqref{eq:full-davie}, the regularity estimates~\eqref{eq:rp_regularity}, and the control estimate~\eqref{eq:control}, we obtain
	\[
		|Y(t)-Y(s)| \lesssim \omega(s,t)^\frac{3}{p} + \omega(s,t)^\frac{2}{p} + \omega(s,t)^\frac{1}{p} \lesssim \omega(s,t)^\frac{1}{p}
	\]
where the constant only depends on the values of $f_\tau^i$ in $K$ for $\tau\in\Fcal_n^2$ and $\|X\|$. It follows from these two estimates that there exists a constant $C_0\in (0,\infty)$ such that
	\[
		\bigg|
			\phi^i(Y(t)) - \phi^i(Y(s)) 
			- g^i_{\ta}(Y(s)) X^{\ta}(s,t) 
			- \frac{1}{2} g^i_{\tab}(Y(s)) X^{\tab}(s,t) 
			- g^i_{\tc}(Y(s)) X^{\tc}(s,t) 
		\bigg| 
		\leq C_0\, \omega(s,t)^\frac{3}{p}.
	\]
It remains to realize that if the above inequality holds for $(s,t)\in\Delta[S,T]$ such that $\omega(s,t)\leq 1$, then it holds, possibly with a different constant, for any $(s,t)\in\Delta[S,T]$. To this end, note that the set $\{(s,t)\in\Delta[S,T]: \omega(s,t)\geq 1\}$ is compact so that the continuous function
	\[ 
		(s,t) \mapsto \omega(s,t)^{-\frac{3}{p}} 
						\bigg| 
							\phi^i(Y(t)) - \phi^i(Y(s)) 
							- g^i_{\ta} X^{\ta}(s,t) 
							- \frac{1}{2} g^i_{\tab}(Y(s)) X^{\tab}(s,t) 
							- g^i_{\tc}(Y(s)) X^{\tc}(s,t) 
						\bigg|
	\]
attains its maximum there. Therefore, if we denote the maximum by $C_i$, replacing constant $C_0$ by $\widetilde{C}:= \max\{C_0, C_1, C_2, \ldots, C_\ell\}$ concludes the proof.
\end{proof}

\begin{remark} 
\label{rem:invariance_of_davie_solutions}
In the setting of \autoref{thm:ch-var}, even if we have that 
	\[
		f^i_{\tab}=0
	\] 
holds for every $i\in\llbracket 1,m\rrbracket$ and every $\alpha,\beta\in \llbracket 1,n \rrbracket$, we get  
	\[ 
		g_{\tab}^i = \frac{\partial^2\phi^i}{\partial y^j\partial y^k} f_{\ta}^jf_{\tb}^k
	\]
for every $i\in\llbracket 1,\ell \rrbracket$ and $\alpha,\beta\in \llbracket 1,n \rrbracket$. Generally, this is not zero. If, however, $X$ is \emph{geometric}, i.e.~if 
	\[ 
		X^{\tc}+X^{\td} = X^{\tab}
	\]
holds for every $\alpha,\beta \in \llbracket 1,n \rrbracket$, we could simplify the resulting expression by grouping together coefficients $g^i_{\tab}$ and $g^i_{\tc}$ and obtain that the estimate 
	\[
		\left|
			\phi^i(Y(t)) - \phi^i(Y(s)) 
			- \widetilde{g}^{\,i}_{\ta}(Y(s)) X^{\ta}(s,t) 
			- \widetilde{g}^{\,i}_{\tc}(Y(s)) X^{\tc}(s,t) 
		\right| 
		\leq \widetilde{C}\omega(s,t)^{\frac{3}{p}},
	\]
where 
	\[
		\widetilde{g}^{\,i}_{\ta} := g^i_{\ta}
			\quad\mbox{and}\quad 
		\widetilde{g}^{\,i}_{\tc} := \frac{1}{2} \left(g_{\tab}^i + g_{\tba}^i\right) + g_{\tc}^i,
	\]
holds for every $i\in\llbracket 1,\ell \rrbracket$ and $(s,t)\in\Delta[S,T]$. It follows that, unless the driving rough path is geometric, Davie solutions are not invariant under smooth changes of coordinates. Therefore, if we wish to arrive at a solution concept with this property, we must consider expansions of the form~\eqref{eq:full-davie}. An immediate consequence is that it will not be possible to express such solutions as a rough integral with respect to $X$, because by construction, rough integrals satisfy an expansion of the form~\eqref{eq:full-davie} with $f^i_{\tab}=0$. Therefore, Picard iteration will not be possible within our framework.
\end{remark} 

\section{Rough velocity fields}
\label{sec:rvf}

In this section, we define a fiber bundle (called the rough velocity bundle) whose sections (called rough velocity fields) will correspond to RDEs in \autoref{sec:RDEs}. To this end, we recall some notions related to jets and jet groups first. We refer to \cite{KolMichSlo93} for the details.

Let $M$ and $N$ be two smooth, finite-dimensional manifolds. (Throughout the paper, we will only consider such manifolds and therefore, we shall omit these two epithets.) Recall that two curves $\gamma,\delta: \Rsf \to M$ are said to have the \emph{second-order contact at zero} if, for every smooth function $\varphi$ on $M$, the difference $\varphi\circ\gamma-\varphi\circ\delta$ and its first and second derivatives are zero at $0\in\Rsf$. In this case, we write $\gamma\sim_2\delta$. Two smooth functions $f,g:M\to N$ are said to \emph{determine the same $2$-jet at $x\in M$} if for every curve $\gamma:\Rsf\to M$ with $\gamma(0)=x$, the curves $f\circ\gamma$ and $g\circ\gamma$ have the second-order contact at zero. An equivalence class of this relation is called a \emph{$2$-jet (of smooth functions) from $M$ into $N$}. Given a smooth function $f:M\to N$, its jet is denoted by $j_x^2f$. The point $x$ is called the \emph{source} of $j_x^2f$ and the value $f(x)$ is called its \emph{target}. 

Let $m,n\in\Nsf$ and denote by $L_{mn}^2$ the set of $2$-jets of smooth functions from $\Rsf^n$ to $\Rsf^m$ with source $0\in\Rsf^n$ and target $0\in\Rsf^m$.\footnote{Note that our notation differs from the one in  \cite{KolMichSlo93}. Namely, in \cite{KolMichSlo93}, $L_{mn}^2$ denotes the set of $2$-jets of smooth functions from $\Rsf^m$ to $\Rsf^n$.} Jets in $L^2_{mn}$ can be identified with their \emph{polynomial representative}, an $m$-tuple of second-order polynomials in $n$ variables with no constant term: Given a jet $a\in L_{mn}^2$, its polynomial representative can be written as 
	\[ 
		a_j^ix^j + \frac{1}{2} a_{jk}^ix^jx^k 
			\quad 
		(i\in\llbracket 1,m \rrbracket;\, j,k\in \llbracket 1,n \rrbracket),
	\]
in which the coefficients satisfy $a_{jk}^i=a_{kj}^i$. The functions $a\mapsto a^i_j$ and $a\mapsto a_{jk}^i$ ($j\leq k$) define a global chart on $L^2_{mn}$. We will simply write $(a_j^i, a_{jk}^i)$ to denote the polynomial representative of $a$. 

Jets with matching source and target can be composed by composing their polynomial representatives and discarding terms with degree larger than two: For $\ell, m, n\in\Nsf$ and jets $a\in L_{\ell m}^2$ and $b\in L_{mn}^2$ with polynomial representatives $(a_j^i, a_{jk}^i)$ and $(b_j^i, b_{jk}^i)$, respectively, the \emph{composition} $c=ab\in L_{\ell n}^2$ is the jet whose polynomial representative $(c_j^i, c_{jk}^i)$ is given by
	\[ 
		c_j^i := a_p^ib_j^p 
			\quad \mbox{and} \quad
		c_{jk}^i := a_p^i b_{jk}^p + a_{pq}^ib_j^pb_k^q
			\quad 
		(i\in \llbracket 1,\ell \rrbracket; \, j,k\in \llbracket 1,n \rrbracket).
	\]

The set of all invertible elements of $L_{mm}^2$ with the jet composition defined above forms a Lie group called the \emph{second-order jet group in dimension $m$} which we denote by $G_m^2$. A jet $a\in L^2_{mm}$ represented by $(a_j^i, a_{jk}^i)$ is invertible if and only if the matrix $(a_j^i)$ is invertible. As the set of invertible matrices is open, the global chart on $L^2_{mm}$ is still a global chart on $G_m^2$. The identity jet is represented by $(a_j^i, a_{jk}^i)=(\delta_j^i, 0)$ where $\delta_j^i$ is the Kronecker delta.

\subsection{Rough velocities}
Recall that for $n\in \Nsf$, the symbol $\Fcal_n^2$ denotes the set of forests of rooted trees labeled by $\llbracket 1,n \rrbracket$ with at most two vertices but excluding the empty forest.

\begin{definition}
For $n\in\Nsf$, denote by $F_n^2$ the real vector space with basis $\Fcal_n^2$ and for $m,n\in\Nsf$ denote by $F_{mn}^2$ the $m$-fold Cartesian product of $F_n^2$. The elements of $F_{mn}^2$ are called \emph{rough velocities}.
\end{definition}

Elements of $F_n^2$ may be seen as polynomials in the variables $\ta$, $\tab$, and $\tc$ with no constant term. Therefore, rough velocities, while being tuples $(f_\tau^i) \in F_{mn}^2$ of real numbers, can be interpreted as $m$-tuples of polynomials in $\ta$, $\tab$, and $\tc$. The change-of-variables formula in \autoref{thm:ch-var} can then be interpreted as an action of a jet on a rough velocity. This action is given by composing two tuples of polynomials---the polynomial representative of the jet and the rough velocity interpreted as an $m$-tuple of polynomials in $\ta$, $\tab$, and $\tc$---and discarding terms that would correspond to forests with more than two vertices. 

\begin{definition}
\label{def:jet-velocity_composition}
For $\ell, m, n\in\Nsf$, a jet $a\in L_{\ell m}^2$ represented by $(a_j^i, a_{jk}^i)$, and a rough velocity $f=(f_\tau^i)\in F_{mn}^2$, the \emph{jet-velocity composition} is defined as $af=(g_\tau^i)\in F_{\ell n}^2$ where
	\[
		g^i_{\ta} := a_j^i f^j_{\ta}, 
			\quad 
		g_{\tab}^i:= a_j^if_{\tab}^j + a_{jk}^i f_{\ta}^j f_{\tb}^k, 
			\quad 
		g_{\tc}^i := a_j^i f^j_{\tc} 
			\quad 
		(i\in\llbracket 1,\ell \rrbracket;\, \alpha,\beta\in \llbracket 1,n \rrbracket).
	\] 
\end{definition}

Clearly, if $a\in G_m^2$ is the identity jet, then we have $af=f$ for any $f\in F_{mn}^2$. It is also clear that the jet-rough velocity composition is a smooth mapping $L^2_{\ell m}\times F_{mn}^2 \to F_{\ell n}^2$. We also have

\begin{lemma}
The jet-rough velocity composition is associative. That is, for $k,\ell, m,n\in\Nsf$, jets $a\in L^2_{k\ell}$, $b\in L^2_{\ell m}$, and a rough velocity $f\in F_{mn}^2$, we have the equality $(ab)f = a(bf)$. 
\end{lemma}

\begin{proof}
	Let $a$ and $b$ be represented by $(a_{r}^i, a_{rs}^i)$ and $(b_{r}^i, b_{rs}^i)$, respectively. Denote by $c:=ab\in L_{km}^2$, $g:=bf\in F_{\ell n}^2$, $h:=ag\in F_{kn}^2$, and $e:=cf\in F_{kn}^2$. Denote also $h=:(h^i_\tau)$ and $e=:(e^i_\tau)$. Let $i\in \llbracket 1,k \rrbracket$ and $\alpha,\beta \in \llbracket 1,n \rrbracket$. Clearly $h^i_{\ta} = e^i_{\ta}$ and $h^i_{\tc} = e^i_{\tc}$. We also have
	\begin{align*}
   		 h^{i}_{\tab}
   		 	& = a^{i}_{r}g^{r}_{\tab} 
   		 		+ a^{i}_{rs}g^{r}_{\ta}g^{s}_{\tb}
  			\\
 			& = a^{i}_{r}\left(b^{r}_{p}f^{p}_{\tab} 
 				+ b^{r}_{pq}f^{p}_{\ta}f^{q}_{\tb}\right) 
 				+ a^{i}_{rs}b^{r}_{p}f^{p}_{\ta}b^{s}_{q}f^{q}_{\tb}
   			\\
   			& = a^{i}_{r}b^{r}_{p}f^{p}_{\tab} 
   				+ \left(a^{i}_{r}b^{r}_{pq} 
   				+ a^{i}_{rs}b^{r}_{p}b^{s}_{q}\right)f^{p}_{\ta}f^{q}_{\tb}
   			\\
   			& = c^{i}_{p}f^{p}_{\tab} 
   				+ c^{i}_{pq}f^{p}_{\ta}f^{q}_{\tb}  
   			\\
			& = g^{i}_{\tab}. 
			\qedhere
  \end{align*}
\end{proof}

\subsection{Rough velocity bundles}
Let $M$ be an $m$-dimensional manifold and let $P^2M$ be the \emph{second-order frame bundle} of $M$, i.e.\ 
the set of $2$-jets of local diffeomorphisms from $\Rsf^m$ to $M$ with source $0\in\Rsf^m$. Recall that $P^2M$ is a \emph{principal fiber bundle} with structure group $G_m^2$: It is a smooth manifold with a smooth projection $P^2M\to M$, $j_0^2\psi \mapsto \psi(0)$, the mapping $P^2M\times G_m^2\to P^2M$, $(j_0^2\varphi, j_0^2\psi)\mapsto j_0^2(\varphi\circ\psi)$, is a smooth right action of the Lie group $G_m^2$ on $P^2M$ that is free and transitive on fibers, and local trivializations exists---each chart $(U,x)$ on $M$ induces a fiber chart $(\widetilde{U},\widetilde{x})$ on $P^2M$ where $\widetilde{U}:=\{j_0^2\psi: \psi(0)\in U\}$ and the chart map $\widetilde{x}: \widetilde{U}\to U\times G_m^2$ is given by $j_0^2\psi\mapsto (\psi(0), j_0^2(x\circ\psi))$. Here, $j_0^2(x\circ\psi)$ is interpreted as an element of $G_m^2$. We refer to \cite[Section 12.12]{KolMichSlo93} for details. Now, as the jet--rough velocity composition is a smooth (left) action of $G_m^2$ on $F_{mn}^2$, one can consider the \emph{associated bundle} to $P^2M$ for this action.

\begin{definition}
For $m,n\in\Nsf$ and an $m$-dimensional manifold $M$, the associated bundle to the second-order frame bundle $P^2M$ for the action of jet-rough velocity composition $G_m\times F_{mn}^2\to F_{mn}^2$ is called the \emph{rough velocity bundle} and denoted by $R^2_nM$.
\end{definition}

Let us recall the main points of the construction. The reader can refer to \cite[Section 10.7]{KolMichSlo93} for details. The underlying set of $R_n^2M$ is the set of orbits of the right action of $G_m^2$ on $P^2M\times F_{mn}^2$ given by $(j_0^2\psi,f)\cdot a= (j_0^2(\psi\circ a), a^{-1}f)$, $(j_0^2\psi,f)\in P^2M\times F_{mn}^2$, $a\in G_m^2$. Its elements are thus equivalence classes of pairs $(j_0^2\psi,f)$ of elements $j_0^2\psi\in P^2M$ and $f\in F_{mn}^2$. The bundle projection $\pi$ maps the class $[(j_0^2\psi,f)]$ to $\psi(0)$. Each chart $(U,x)$ on $M$ induces a fiber chart $(\overline{U},\overline{x})$ on $R^2_nM$ where $\overline{U}:=\{ [(j_0^2\psi,f)]: \psi(0)\in U\}$ and where the chart map $\overline{x}: \overline{U}\to U\times F_{mn}^2$ is given by $[(j_0^2\psi,f)]\mapsto (\psi(0),af)$ with $a= j_0^2(x\circ \psi)$ interpreted as an element of $G_m^2$. By composing this fiber chart with $x$ in the first component, we obtain a chart on $R^2_nM$ with domain $\overline{U}$ and components $(x^i, f_\tau^i)$ for $i\in \llbracket 1,m \rrbracket$ and $\tau\in \Fcal_n^2$. Whenever we perform local calculations on $R_n^2M$, we shall do so using these induced charts.

Given two charts $(U,x)$ and $(V,y)$ in $M$, let $a_{xy}:U\cap V\to G_m^2$ be the map defined by $p\mapsto j_{y(p)}^2 x\circ y^{-1}$, i.e.\ for $p\in U\cap V$, $a_{xy}(p)$ is the $2$-jet represented by $(a_j^i, a_{jk}^i)$ where
	\[
		a_j^i = \frac{\partial x^i}{\partial y^j}(y(p)) 
			\quad \mbox{and} \quad 
		a^i_{jk} = \frac{\partial^2 x^i}{\partial y^j y^k} (y(p)) 
			\quad 
		(i,j,k\in \llbracket 1,m \rrbracket).
	\]
Let $(\overline{U}, \overline{x})$ and $(\overline{V},\overline{y})$ be the fiber charts on $R_n^2M$ induced by $(U,x)$ and $(V,y)$, respectively. The transition function between them is the mapping 
	\[
		\overline{x}\circ\overline{y}^{-1}: (U\cap V)\times F_{mn}^2\to(U\cap V)\times F_{mn}^2, 
			\quad 
		(p,f)\mapsto (p,a_{yx}(p)f).
	\]
This is our transformation law for the coefficients of solutions to RDEs. 
	
The rule associating $M\to R^2_nM$ is in fact a functor on the category of smooth manifolds, see \cite[Section 12.12]{KolMichSlo93}. Let us describe its action on morphisms. To this end, let $M$ and $M'$ be manifolds of dimension $m$ and $m'$, respectively. A smooth mapping $\phi: M\to M'$ maps to a fiber bundle morphism $R_n^2\phi: R_n^2M\to R_n^2M'$ given by $[(j_0^2\psi,f)]\mapsto [(j_{\psi(0)}^2(\phi\circ\psi),f)]$. For $\bm{f} = [(j_0^2\psi,f)]\in R_n^2M$ with image $\bm{g}=[(j_{\psi(0)}^2(\phi\circ\psi),f)]\in R_n^2M'$, let $p:=\psi(0)$ and let $(U,x)$ be a chart around $p$ and $(V,y)$ be a chart around $\phi(p)$. Denote the induced coordinates of $\bm{f}$ by $(x^i, f_\tau^i)$ and the induced coordinates of $\bm{g}$ by $(y^i, g_\tau^i)$. Then
	\[
		y^i  = \phi^i(p)
	\]
and 
	\[
		g_{\ta}^i = \frac{\partial \phi^i}{\partial x^j}(p)f_{\ta}^j, 
			\quad 
		g_{\tab}^i = \frac{\partial \phi^i}{\partial x^j}(p) f^j_{\tab} 
						+ \frac{\partial^2 \phi^i}{\partial y^j\partial y^k}(p) f^j_{\ta} f^k_{\tb}, 
			\quad 
		g^i_{\tc} = \frac{\partial \phi^i}{\partial x^j}(p) f^j_{\tc}
	\]
hold for every $i\in \llbracket 1,m' \rrbracket$ and every $\alpha,\beta\in \llbracket 1,n \rrbracket$. Due to the functoriality of the construction, we obtain the following

\begin{lemma}
\label{thm:r2-functorial}
Let $M, M'$, and $M''$ be manifolds and let $\phi: M\to M'$ and $\psi: M'\to M''$ be smooth maps. Then $R_n^2(\psi\circ\phi) = R^2_n\psi \circ R^2_n\phi$. 
\end{lemma}

We will be interested in the structure of $R_n^2M$ when $M$ is an embedded submanifold of manifold $M'$. Let us prove that in this situation, $R_n^2M$ is an embedded submanifold of $R_n^2M'$. 

\begin{lemma}
\label{thm:r2-preserve-embedding}
Let $M, M'$ be manifolds and let $\phi: M\to M'$ be an embedding. Then the mapping $R_n^2\phi:R_n^2M\to R_n^2M'$ is also an embedding.
\end{lemma}

\begin{proof}
Denote by $m$ and $m'$ the dimensions of manifolds $M$ and $M'$, respectively. Let $p\in M$ be given. By the rank theorem (see, e.g., \cite[Theorem 4.12]{Lee12}), there exist charts $(U,x)$ around $p$ and $(V,y)$ around $\phi(p)$ such that $\phi(U)\subseteq V$ and in which the coordinate representation of $\phi$ is $(x^1,x^2,\ldots, x^m)\mapsto (x^1, x^2, \ldots, x^m, 0,0, \ldots, 0)$. Then the coordinate representation of $R_n^2\phi$ in the induced charts on $R_n^2M$ and $R_n^2M'$ maps $(x^i, f_\tau^i)$ to $(y^i, g_\tau^i)$ that is, for $i\in \llbracket 1, m' \rrbracket$  and $\tau\in\Fcal_n^2$ given by 
	\[
		y^i = \begin{cases}
				x^i, & \quad \mbox{if } i\leq m,\\
				0, & \quad \mbox{otherwise},
			   \end{cases}
		\qquad 
		g^i_{\tau} = \begin{cases}
						f_\tau^i, & \quad \mbox{if } i\leq m,\\
						0, & \quad\mbox{otherwise}.
					 \end{cases}
	\]
It follows that $R_n^2\phi$ is an immersion. Moreover, if $R_n^2\phi(\bm{f}) = R_n^2\phi(\bm{g})$ for $\bm{f},\bm{g}\in R_n^2M$, then $\phi(\pi\bm{f}) = \pi (R_n^2\phi(\bm{f})) = \pi(R_n^2\phi(\bm{g})) = \phi (\pi(\bm{g}))$ so that we have, by injectivity of $\phi$, that $\pi(\bm{f})=\pi(\bm{g})$ and it follows from the above coordinate representation taken around $p=\pi(\bm{f})$ that $\bm{f}=\bm{g}$. In other words, $R_n^2\phi$ is injective. Finally, because $\phi$ is an embedding, we may choose the charts $(U,x)$ and $(V,y)$ in such a way that $\phi(M)\cap V = \phi(U)$. Then $R_n^2\phi(R_n^2M)\cap \pi^{-1}(V) = R_n^2\phi (\pi^{-1}(U))$ which is a slice in $\pi^{-1}(V)$. This verifies the local slice condition (see, e.g., \cite[Theorem 5.8]{Lee12}) and it follows that the image of $R_n^2\phi$ is an embedded submanifold.
\end{proof}

\subsection{Rough velocity fields}

Let us now define the notion of rough velocity fields. We will see in \autoref{sec:RDEs} that they correspond to RDEs.

\begin{definition}
\label{def:rvf}
Let $n\in\Nsf$ and let $M$ be a manifold. A \emph{rough velocity field} is a smooth section of $R_n^2M$, i.e.\ a smooth function $f:M\to R_n^2M$ such that $(\pi\circ f)(p) = p$ holds for every $p\in M$. 
\end{definition}

\begin{remark}
Let $m,n\in\Nsf$, let $M$ be an $m$-dimensional manifold, and let $(U,x)$ be a chart on~$M$. If $f: M\to R_n^2M$ is a rough velocity field, then it has the coordinate representation $p\mapsto (p^i, f_\tau^i(p))$ for some smooth functions $f_\tau^i:U\to\Rsf$ ($i\in \llbracket 1, m \rrbracket$, $\tau\in \Fcal_n^2$) in the induced chart $(\overline{U},\overline{x})$ on $R_n^2M$.
\end{remark}

Let $m,n\in\Nsf$ and let $M$ be an $m$-dimensional manifold. The element $0\in F_{mn}^2$ is invariant under the action of $G_m^2$ and therefore, each fiber $\pi^{-1}(p)$ of $R_n^2M$ has a distinguished element whose coordinate representation in any fiber chart is $(p,0)$. We will denote this element $0_p$. 

\begin{definition}
\label{def:support-rvf}
Let $n\in\Nsf$, let $M$ be a manifold, and let $f: M\to  R_n^2M$ be a rough velocity field. By the \emph{support} of $f$ we mean the closure of the set $\{p\in M: f(p)\neq 0_p\}$.
\end{definition}

Similarly to vector fields, we have a notion of relatedness which allows us to transform solutions to RDEs between manifolds.

\begin{definition}
\label{def:related-rvf}
Let $n\in\Nsf$, let $M$ and $M'$ be manifolds, $\phi: M\to M'$ be a smooth map, $f:M\to R_n^2M$ be a rough velocity field on $M$, and $g: M'\to R_n^2M'$ be a rough velocity field on $M'$. We say $g$ \emph{$\phi$-related} to $f$ if $g(\phi(p))=R^2_n\phi(f(p))$ holds for all $p\in M$. 
\end{definition}

In \autoref{sec:M-existence-uniqueness}, we will prove the existence and uniqueness of solutions to RDEs on manifolds. This will be done by gluing solutions constructed in coordinate patches. To make such a procedure feasible, we will need a mechanism that  allows the restriction of rough velocity fields to coordinate patches. This is not trivial because $R_n^2M$ is not a vector bundle and we cannot simply multiply rough velocity fields by bump functions. In what follows, we describe a suitable operation.

\begin{definition}
Let $m,n\in\Nsf$, let $M$ be an $m$-dimensional manifold, and let $(U,x)$ be a chart on $M$. Let $f: M\to R_n^2M$ be a rough velocity field whose representation in the induced chart $(\overline{U},\overline{x})$ on $R_n^2M$ is $(f^i_\tau)$. Let also $\delta: M\to\Rsf$ be a smooth function. A \emph{rescaling} of $f$ by $\delta$ is the rough velocity field $\widetilde{f}: M\to R_n^2M $ whose representation $(\widetilde{f}^i_\tau)$ in chart $(\overline{U},\overline{x})$ is given by 
	\begin{equation}
	\label{eq:rescaling}
		\widetilde{f}^i_{\ta} = \delta f^i_{\ta}, 
			\quad 
		\widetilde{f}^i_{\tab} = \delta^2f^i_{\tab},
			\quad \mbox{and} \quad 
		\widetilde{f}^i_{\tc} = \frac{\partial \delta}{\partial x^j}f^j_{\ta}\delta f^i_{\tb} + \delta^2f^i_{\tc}
	\end{equation}
for $i\in \llbracket 1,m \rrbracket$ and $\alpha,\beta\in\llbracket 1,n \rrbracket$.
\end{definition}

\begin{remark}
Rescaling of a rough velocity field by a smooth function is an intrinsic notion: Let $(U,x)$ and $(V,y)$ be charts on manifold $M$ and let $(f^i_\tau)$ and $(g^i_\tau)$ be the representations of a rough velocity field $M\to R_n^2M$ in the induced charts $(\overline{U},\overline{x})$ and $(\overline{V},\overline{y})$ on $R_n^2$, respectively. Then on $U\cap V$, we have 
	\[
		\widetilde{g}^i_{\ta} 
			= \frac{\partial y^i}{\partial x^j}\widetilde{f}^j_{\ta} 
			= \frac{\partial y^i}{\partial x^j} \delta f^j_{\ta} 
			= \delta g^i_{\ta}
	\]
and, similarly, 
	\[
 		\widetilde{g}^{i}_{\tab} 
 	 		= \frac{\partial y^{i}}{\partial x^{j}}\widetilde{f}^{j}_{\tab} 
 	 			+ \frac{\partial^{2} y^{i}}{\partial x^{j}\partial x^{k}}\widetilde{f}^{j}_{\ta}\widetilde{f}^{k}_{\tb}
  			= \delta^{2}
				\left(
 					\frac{\partial y^{i}}{\partial x^{j}}f^{j}_{\tab} 
 						+ \frac{\partial^{2} y^{i}}{\partial x^{j}\partial x^{k}}f^{j}_{\ta}f^{k}_{\tb}
  				\right)
  			= \delta^{2}g^{i}_{\tab},
	\]
and 
	\[
  		\widetilde{g}^{i}_{\tc} 
  			= \frac{\partial y^{i}}{\partial x^{k}}\widetilde{f}^{k}_{\tc}
 			= \frac{\partial y^{i}}{\partial x^{k}}
  				\left(
 					\frac{\partial \delta}{\partial x^{j}}f^{j}_{\ta}\delta f^{k}_{\tb} 
 						+ \delta^{2}f^{k}_{\tc}
 				 \right)
  			= \frac{\partial y^{i}}{\partial x^{k}}
 				\left(
  					\frac{\partial \delta}{\partial y^{\ell}}\frac{\partial y^{\ell}}{\partial x^{j}}f^{j}_{\ta}\delta f^{k}_{\tb} 
 					 	+ \delta^{2}f^{k}_{\tc}
				\right)
  			= \frac{\partial \delta}{\partial y^{\ell}}g^{\ell}_{\ta}\delta g^{i}_{\tb} 
  				+ \delta^{2}g^{i}_{\tc}
	\]
for every $i\in \llbracket 1,m \rrbracket$ and $\alpha,\beta\in\llbracket 1,n\rrbracket$.
\end{remark}

If $M=\Rsf^m$, we may identify a rough velocity field $f$ with its global coordinate representation and view it as a tuple $(f_{\tau}^i)$ of functions $f_\tau^i:\Rsf^m\to\Rsf$ ($i\in \llbracket 1,m\rrbracket$, $\tau\in\Fcal_n^2$). In this way, it will be possible to speak about its $C^k$-norm.

\begin{definition}
Let $m,n\in\Nsf$ and let $f=(f^i_\tau)$ be the representation of a rough velocity field $f: \Rsf^m\to R_n^2\Rsf^m$ in the global chart. For $k\in\Nsf_0$, the $C^k$-norm of $f$ is then defined by
	\[
		\|f\|_k := \sum_{i=1}^m \sum_{\tau\in \Fcal_n^2} \sum_{|\alpha|\leq k} \sup_{x\in\Rsf^m} |D^\alpha f_{\tau}^i (x)|
	\]
where we use the standard multiindex notation (see, e.g., \cite{Tri92}).
\end{definition}

Finally, we give examples of rough velocity fields that arise naturally from a vector field.

\begin{example}
\label{ex:canonical-rvf}
Let $m,n\in\Nsf$ and let $f^i_{\alpha}:\Rsf^m\to\Rsf$ (for $i\in \llbracket 1,m \rrbracket$ and $\alpha\in \llbracket 1,n \rrbracket$) be smooth functions. Then we can define a rough velocity field $f: \Rsf^m\to R_n^2\Rsf^m$ by defining its coordinate representation in the global chart on $\Rsf^m$ as
\begin{equation}
\label{eq:canonical-rvf} 
    f^i_{\ta}:= f^i_\alpha, 
        \quad 
    f^i_{\tc}:= \frac{\partial f^i_\beta}{\partial x^j}f^j_{\alpha}
        \quad 
    f^i_{\tab}:= 0, 
\end{equation}
for $i\in \llbracket 1,m \rrbracket$ and $\alpha,\beta\in \llbracket 1,n \rrbracket$. Solutions of RDEs corresponding to rough velocity fields of this form are solutions in the Davie sense; see \autoref{def:davie-solution}.
\end{example}

\begin{example}
\label{ex:connection-rvf}
More generally, we can define a rough velocity field as follows: Let $m,n\in\Nsf$, $M$ be an $m$-dimensional manifold, $\nabla$ be a torsion-free affine connection on $M$, $f_{\alpha}$ (for $\alpha\in \llbracket 1,n \rrbracket$) be smooth vector fields on $M$, and let $(U, x)$ be a chart on $M$. On $U$, define
\begin{equation*}
    f^{i}_{\ta} := f^{i}_{\alpha},
        \quad
    f^{i}_{\tc} := f^{j}_{\alpha}\left(\frac{f^{i}_{\beta}}{\partial x^{j}} + \Gamma^{i}_{jk}f^{k}_{\beta}\right),
        \quad
    f^{i}_{\tab} := -\Gamma^{i}_{jk}f^{j}_{\alpha}f^{k}_{\beta},
\end{equation*}
for $i\in \llbracket 1,m \rrbracket$, $\alpha,\beta\in \llbracket 1,n \rrbracket$, where $(f^i_\alpha)$ is the coordinate representation of $f_\alpha$ in the chart $(U,x)$ and where $\Gamma^{i}_{jk}$ are the Christoffel symbols given by
\begin{equation*}
    \nabla_{\frac{\partial}{\partial x^{j}}}\frac{\partial}{\partial x^{k}} = \Gamma^{i}_{jk}\frac{\partial}{\partial x^{i}}.
\end{equation*}
A routine calculation shows that the above expression for $(f^i_\tau)$ is independent of the chosen chart and thus defines a rough velocity field $f: M\to R_n^2M$.
\end{example}

\section{RDEs on manifolds}
\label{sec:RDEs}

In this section, we define a concept of solutions to RDEs---the so-called integral curves. These generalize Davie solutions and are invariant with respect to coordinate changes which allows us to consider RDEs on (smooth, finite-dimensional) manifolds driven by branched rough paths. We describe the class of RDEs---the admissible rough velocity fields---that will be solvable within our framework and give the local existence and uniqueness results. We also give a sufficient condition for the existence of a global solution. 

\subsection{Integral curves}
\label{sec:integral-curves}

Initially, we define a suitable concept of solutions to RDEs on manifolds that will be invariant under smooth changes of coordinates. As we have seen in \autoref{rem:invariance_of_davie_solutions}, Davie solutions that satisfy~\eqref{eq:davie} do not, in general, have this property. Therefore, we will impose condition~\eqref{eq:full-davie} instead. To simplify notation, let us introduce the \emph{symmetry factor} $\sigma(\tau)$ defined as the number of label-preserving automorphisms of $\tau \in\Fcal_n^2$. Explicitly, we have 
	\[
		\sigma(\ta) = 1, 
			\quad 
		\sigma(\tab) = \begin{cases} 
							1, & \quad \alpha\neq \beta,\\
							2, & \quad \alpha=\beta, 
						\end{cases} 
			\quad\mbox{and}\quad 
		\sigma(\tc) = 1
	\]
for $\alpha,\beta\in \llbracket 1,n \rrbracket$. The notion of an integral curve is defined now.

\begin{definition}
\label{def:integral-curve}
Let $m,n\in\Nsf$, let $[S,T]\subseteq\Rsf$ be an interval, let $p\in [1,3)$, and let $\omega$ be a control function on the interval $[S,T]$. Let $X=(X^\tau)$ be an $n$-dimensional branched rough path on $[S,T]$ with regularity $p$ and control $\omega$. Let $M$ be an $m$-dimensional manifold and let $f: M\to R_n^2M$ be a rough velocity field. We say that a function $Y: [S,T]\to M$ is an \emph{integral curve} of the rough velocity field $f$ driven by the rough path $X$ if it is continuous and if for every compact interval $I\subseteq [S,T]$ and every chart $(U,x)$ around $Y(I)$ there exists $C\in (0,\infty)$ such that the inequality
		\begin{equation}
		\label{eq:integral-curve}
			\left|
				Y^i(t) - Y^i(s) 
				- \sum_{\tau\in \Fcal_n^2} \frac{1}{\sigma(\tau)} f_{\tau}^i(Y(s)) X^{\tau}(s,t) 
			\right| 
			\leq C\omega(s,t)^\frac{3}{p}
		\end{equation}
holds for every $i\in \llbracket 1,m \rrbracket$ and $(s,t)\in\Delta I$. Here, $Y^i=x^i\circ Y$ ($i\in \llbracket 1,m \rrbracket$) are the components of the coordinate representation of $Y$ in chart $(U,x)$ and $f_\tau^i$ ($i\in \llbracket 1,m \rrbracket$ and $\tau\in\Fcal_n^2$) are the components of the representation of $f$ in the induced chart $(\overline{U},\overline{x})$ on $R_n^2M$.
\end{definition}

The following example relates integral curves to Davie solutions to RDEs.

\begin{example}
Let $m,n\in\Nsf$, let $[S,T]\subseteq\Rsf$ be an interval, let $p\in [1,3)$, $x\in C^{p-\mathrm{var}}([S,T],\Rsf^n)$, and let $f_\alpha^i:\Rsf^m\to\Rsf$ (for $i\in \llbracket 1,m \rrbracket$ and $\alpha\in \llbracket 1,n \rrbracket$) be smooth functions. Let also $Y_S\in\Rsf^m$. Consider the system of controlled differential equations~\eqref{eq:ode}. Let $X=(X^\tau)$ be a rough path lift of function $x$ (see \autoref{ex:natural-lift}) and let $f=(f^i_\tau)$ be the rough velocity field on $\Rsf^m$ defined by~\eqref{eq:canonical-rvf}. Then a function $Y:[S,T]\to\Rsf^m$ is a Davie solution to problem~\eqref{eq:ode} if and only if $Y$ is an integral curve of $f$ driven by $X$ satisfying $Y(S)=Y_S$. 
\end{example}

The next example demonstrates the need for a different constant on each compact interval $I$.

\begin{example}
Let $X=(X^\tau)$ be the (scalar) branched rough path with regularity $p=1$ and control $\omega(s,t):=t-s$, $(s,t)\in \Delta[0,1]$, built from the smooth path $x(t)=t$ on the interval $[0,1]$ as in \autoref{ex:natural-lift}. Let $M:=\Rsf$ and consider the rough velocity field $(f_\tau)$ given by 
	\[
		f_{\tao}(y):=1, 
			\quad 
		f_{\taboo}(y):=0, 
			\quad 
		f_{\tcoo}(y):=0
	\]
for $y\in\Rsf$ in the global chart. One would naturally expect the path $Y(t) := t$, $t\in [0,1]$, to be an integral curve of $f$ driven by $X$. In the standard chart, we have 
	\[
		Y(t) - Y(s) - \sum_{\tau\in\Fcal_1^2} \frac{1}{\sigma(\tau)} f_{\tau}(Y(s)) X^{\tau}(s,t) 
			= t-s - 1(t-s) 
			= 0
	\]
for $(s,t)\in \Delta[0,1]$. In the chart $((0,\infty), \log)$, the rough velocity field $f$ is given by 
	\[
		f_{\tao}(y) = y^{-1}, 
			\quad 
		f_{\taboo}(y) = -y^{-2}, 
			\quad 
		f_{\tcoo}(y) = 0
	\]
for $y\in (0,\infty)$ and so
	\[
		Y(t)-Y(s) - \sum_{\tau\in\Fcal_1^2} \frac{1}{\sigma(\tau)} f_{\tau}(Y(s)) X^{\tau}(s,t) 
			= \log(t) - \log(s) - \frac{1}{s}(t-s) + \frac{1}{2s^2}(t-s)^2
	\]
for $(s,t)\in \Delta(0,1]$. If $s\in (0,1/2]$ and $t=2s$, we obtain 
	\[
		\omega(s,t)^{-\frac{3}{p}} 
		\left|
			Y(t)-Y(s) 
			- \sum_{\tau\in\Fcal_1^2} \frac{1}{\sigma(\tau)} f_{\tau}(Y(s)) X^{\tau}(s,t)
		\right| 
		= \frac{\log (2)-\frac{1}{2}}{s^3}
	\]
which grows without bound as $s\to 0+$. 
\end{example}

Being an integral curve of a rough velocity field driven by a rough path is a local property. It suffices to check~\eqref{eq:integral-curve} in one chart in a neighborhood of each $r\in [S,T]$. To prove this claim, we first give the following

\begin{lemma}
\label{thm:integral-curve-delta}
Let $m,n\in\Nsf$, $I\subseteq\Rsf$ be a compact interval, $p\in[1,3)$, and $\omega$ a control function on interval $I$. Let $X=(X^\tau)$ be an $n$-dimensional branched rough path on $I$ with regularity $p$ and control $\omega$. Let $M$ be an $m$-dimensional manifold and let $f: M\to R_n^2M$ be a rough velocity field. Let $Y:I\to M$ be a continuous function, $\widetilde{C}\in (0,\infty)$, $\delta\in (0,\infty)$, and let $(U,x)$ be a chart around $Y(I)$ such that inequality~\eqref{eq:integral-curve} with $\widetilde{C}$ in place of $C$ holds for all for every $i\in \llbracket 1,m \rrbracket$ and every $(s,t)\in\Delta I$ satisfying $\omega(s,t)<\delta$. Then there exists a constant $C\in (0,\infty)$ such that inequality~\eqref{eq:integral-curve} holds for every $i\in \llbracket 1,m \rrbracket$ and every $(s,t)\in \Delta I$ .
\end{lemma}

\begin{proof}
The set $\{ (s,t)\in \Delta I: \omega(s,t)\geq \delta\}$ is compact so that for $i\in\llbracket 1,m \rrbracket$, the continuous function 
	\[ 
		(s,t) \mapsto \omega(s,t)^{-\frac{3}{p}} 
							\left| 
								Y^i(t) - Y^i(s) 
								- \sum_{\tau\in\Fcal_n^2} \frac{1}{\sigma(\tau)} f_\tau^i(Y(s))X^\tau(s,t)  
							\right| 
	\]
attains its maximum there. If the maximum is $\widetilde{C}_i$, set $C:=\max\{\widetilde{C}, \widetilde{C}_1,\widetilde{C}_2,\ldots, \widetilde{C}_m\}$.
\end{proof}

\begin{lemma}
\label{thm:int-curve-local}
Let $m,n\in\Nsf$, $[S,T]\subseteq\Rsf$ be an interval, $p\in[1,3)$, and $\omega$ a control function on interval $[S,T]$. Let $X=(X^\tau)$ be an $n$-dimensional branched rough path on $[S,T]$ with regularity $p$ and control $\omega$. Let $M$ be an $m$-dimensional manifold and let $f: M\to R_n^2M$ be a rough velocity field. Let $Y:[S,T]\to M$ be a continuous function such that for every $r\in [S,T]$ there exists
\begin{enumerate}[label=(\alph*)]
\itemsep0em
    \item an interval $[u,v] \subset [S,T]$ that contains $r$ in its relative interior and such that we either have $v = T$ or $\omega(r,v) > 0$,
    \item a chart $(U, x)$ whose domain contains the image of $[u,v]$ under $Y$, and
    \item a constant $C\in (0,\infty)$
\end{enumerate}
such that inequality~\eqref{eq:integral-curve}, where $(Y^i)$ is the representation of $Y$ in the chart $(U,x)$ and $(f^i_\tau)$ is the representation of $f$ in the induced chart $(\overline{U},\overline{x})$, holds for every $i\in \llbracket 1,m \rrbracket$ and $(s,t)\in\Delta [u,v]$. Then $Y$ is an integral curve of the rough velocity field $f$ driven by the rough path $X$.
\end{lemma}

\begin{proof}
Let $I\subseteq [S,T]$ be a compact interval and let $(V,y)$ be a chart around $Y(I)$. For every $r \in [S,T]$, let $[u_r, v_r]$,  $(U_r,x_r)$ and $C_r \in (0, \infty)$ be as in the statement of the lemma. Denote $\delta_r = \omega(r, v_r) / 2$. We can shrink $U_r$ to ensure that $U_r\subseteq V$ and that $x_r(U_r)$ is a convex set. We can further find $\tilde{v}_r \in (r, v_r)$ such that $\omega(r, \tilde{v}_r) = \omega(\tilde{v}_r, v_r) = \delta_r$ and such that the relative interiors of $[u_r, \tilde{v}_r]$ cover $[S,T]$. Then there exist a finite number of points, say $r_1,r_2, \ldots, r_K\in I$, such that $I\subseteq \bigcup_{k=1}^K [u_{r_k}, \tilde{v}_{r_k}]$.
Let $\delta = \min \{\delta_{r_k} : v_k \neq T \}$. Then $\delta > 0$. Let $(s,t) \in \Delta I$ be such that $\omega(s,t) < \delta$. Then there exists $k \in \llbracket 1, K \rrbracket$ be such $s \in [u_{r_k}, \tilde{v}_{r_k}]$. Either $v_{r_k} = T$ or $v_{r_k} < T$. In the first case, clearly $t \in [u_{r_k}, v_{r_k}]$. In the second case, we must also have $t \in [u_{r_k}, v_{r_k}]$, for if $t > v_{r_k}$, then we would have $\delta \le \delta_{r_k} = \omega(\tilde{v}_{r_k}, v_{r_k}) \le \omega(s, t) < \delta$.

Applying \autoref{thm:ch-var} to the smooth function $y\circ x^{-1}_{r_k}$ and to the representation of the path $Y$ restricted to $[u_{r_k}, v_{r_k}]$ in chart $(U_{r_k},x_{r_k})$ yields a constant $\tilde{C}_{k}\in (0,\infty)$ such that 
\begin{equation*}
    \left|
        Y^i(t) - Y^i(s)
        - \sum_{\tau\in\Fcal_n^2} \frac{1}{\sigma(\tau)} f_\tau^i(Y(s)) X^\tau(s,t)
    \right|
    \leq \tilde{C}_k \omega(s,t)^\frac{3}{p}
\end{equation*}
holds for every $i\in\llbracket 1,m\rrbracket$ and $(s,t)\in\Delta [u_{r_k}, v_{r_k}]$, where $(Y^i)$ is the coordinate representation of $Y$ in chart $(V,y)$ and $(f^i_\tau)$ is the coordinate representation of $f$ in the induced chart $(\overline{V}, \overline{y})$ on $R_n^2M$. Appealing to \autoref{thm:integral-curve-delta} (with $\tilde{C}:= \max_{k\in \llbracket 1,\tilde{K}\rrbracket} \tilde{C}_k$) completes the proof.
\end{proof}

\begin{lemma}
\label{thm:related-fields}
Let $m,m',n\in\Nsf$, let $[S,T]\subseteq\Rsf$ be an interval, let $p\in [1,3)$, and let $\omega$ be a control function on interval $[S,T]$. Let $X$ be an $n$-dimensional branched rough path on $[S,T]$ with regularity $p$ and control $\omega$. Let $M, M'$ be two manifolds of dimensions $m$ and $m'$, respectively, let $\phi: M\to M'$ be a smooth map, and let $f:M\to R_n^2M$ be a rough velocity field on $M$ and $g: M'\to R_n^2M'$ be a rough velocity field on $M'$ that is $\phi$-related to $f$. If $Y:[S,T]\to M$ is an integral curve of $f$ driven by $X$, then $\phi\circ Y:[S,T]\to M'$ is an integral curve of $g$ driven by $X$. 
\end{lemma}

\begin{proof}
Let $r\in [S,T]$ and let $(U,x)$ be a chart around $Y(r)$ and $(V,y)$ be a chart around $\phi(Y(r))$ such that $x(U)$ is convex and $\phi(U)\subseteq V$. Let $\widetilde{U}$ be a neighborhood of $Y(r)$ whose closure is contained in $U$ and set
\begin{align*}
    u = \sup\left\{s \in [S, r] : Y(s) \notin \widetilde{U}\right\} \vee S,
    \\
    v = \inf\left\{t \in [r, T] : Y(t) \notin \widetilde{U}\right\} \wedge T.
\end{align*}
Then $u<v$, $r$ is contained in the relative interior of $[u,v]$, and $Y([u,v])\subseteq U$. Since $Y$ is an integral curve of $f$ driven by $X$, there exists a constant $C\in (0,\infty)$ such that the inequality
	\[
		\left|
			Y^i(t) - Y^i(s) 
			- \sum_{\tau\in\Fcal_n^2} \frac{1}{\sigma(\tau)} f^i_\tau(Y(s)) X^\tau(s,t)
		\right| 
		\leq C\omega(s,t)^\frac{3}{p},
	\]
where $(Y^i)$ is the representation of $Y: [u,v]\to M$ in chart $(U,x)$ and $(f^i_\tau)$ is the representation of $f$ in the induced chart $(\overline{U},\overline{x})$, holds for every $i\in\llbracket 1,m \rrbracket$ and $(s,t)\in\Delta [u,v]$. Consequently, if $\omega(r,v) = 0$, we have $Y(v) = Y(r) \in \widetilde{U}$, and thus $v = T$. As $g$ is $\phi$-related to $f$, by applying \autoref{thm:ch-var} to the smooth function $y\circ\phi\circ x^{-1}$ and the representation $(Y^i)$, we see that there exists a constant $\widetilde{C}\in (0,\infty)$ such that the inequality
	\[
		\left|
			\phi^i(Y(t)) - \phi^i(Y(s)) 
			- \sum_{\tau\in \Fcal_n^2} \frac{1}{\sigma(\tau)} g^i_\tau(Y(s))X^{\tau}(s,t)
		\right| 
		\leq \widetilde{C}\omega(s,t)^\frac{3}{p},
	\]
where $(\phi^i(Y))$  is the representation of $\phi\circ Y$ in chart $(V,y)$ and $(g^i_\tau)$ is the representation of $g$ in the induced chart $(\overline{V},\overline{y})$, holds for every $i\in \llbracket 1,m' \rrbracket$ and $(s,t)\in \Delta [u,v]$. Thus, by \autoref{thm:int-curve-local}, $\phi\circ Y$ is an integral curve of $g$ driven by $X$.
\end{proof}

\subsection{Admissible rough velocity fields}
\label{sec:admissible-rvf}

Not all rough velocity fields can be expected to have integral curves for all drivers. Consider the following

\begin{example}
Let $m,n\in\Nsf$ and let $x: [0,T]\to\Rsf^n$ be a smooth path. Let $(X^{\tau})$ be the canonical rough path lift with regularity $p=1$ and control $\omega(s,t):= t-s$ (for $(s,t)\in\Delta [0,T]$) as in \autoref{ex:natural-lift}. Suppose that $f: \Rsf^m\to R_n^2\Rsf^m$ is a rough velocity field with components $(f^i_\tau)$ in the standard chart. If $Y:[0,T]\to\Rsf^m$ is an integral curve of $f$ driven by $X$, then it follows that there is a constant $C\in (0,\infty)$ such that the estimate
	\[
		\left| 
			Y^i(t)- Y^i(s) - f^i_{\ta}(Y(s)) \dot{x}^\alpha(s) (t-s)
		\right|
		\leq C(t-s)^2
	\]
holds for every $i\in\llbracket 1,m \rrbracket$ and $(s,t) \in \Delta[0,T]$ such that $|t-s|$ is small.\footnote{Here, we denote by $\dot{x}$ the derivative of a function $x$.} This implies that, for $i\in \llbracket 1,m \rrbracket$, the equation
	\[
		\dot{Y}^i(s) = f^i_{\ta} (Y(s)) \dot{x}^\alpha(s)
	\]	
holds for every $s\in (0,T)$ so that $Y$ is the solution to an ODE with smooth coefficients and hence is itself smooth. But from~\eqref{eq:integral-curve} we see that the first and second partial derivative of 
	\[
		Y^i(t) - Y^i(s) - \sum_{\tau\in\Fcal_n^2}\frac{1}{\sigma(\tau)} f_\tau^i(Y(s)) X^\tau(s,t)
	\]
at the diagonal $\{s=t\}$ must be zero. By calculating the derivatives, we obtain that 
	\begin{align*} 
		0 & = \dot{Y}^i(s) - f_{\ta}^i(Y(s)) \dot{x}^\alpha(s),\\
		0 & = \frac{\partial f^i_{\ta}}{\partial y^j}(Y(s)) \dot{Y}^j(s) \dot{x}^\alpha(s) 
				- f^i_{\tab}(Y(s)) \dot{x}^\alpha(s)\dot{x}^\beta(s) 
				- f^i_{\tc}(Y(s)) \dot{x}^\alpha(s)\dot{x}^\beta(s)
	\end{align*}
holds for every $s\in (0,T)$. Substituting $\dot{Y}^i(s)$ from the first equation to the second one yields
	\[
		0 = \left(
				\frac{\partial f^i_{\ta}}{\partial y^j}(Y(s)) f^j_{\tb}(Y(s)) 
				- f^i_{\tab}(Y(s)) 
				- f^i_{\tc}(Y(s))
			\right)
			\dot{x}^\alpha(s) \dot{x}^\beta(s).
	\]
Therefore, in order for a rough velocity field to have integral curves for all smooth drivers through all points, the expression in the brackets above must equal to zero.
\end{example}

Thus we arrive at the notion of admissibility. 

\begin{definition}
\label{def:admissible-rvf}
Let $m,n\in \Nsf$ and let $M$ be an $m$-dimensional manifold. A rough velocity field $f:M\to R_n^2M$ is called \emph{admissible} if, for every $p\in M$, there exists a chart $(U,x)$ around $p$ such that the coordinate representation $(f_\tau^i)$ of $f$ in the induced chart $(\overline{U},\overline{x})$ on $R_n^2M$ satisfies 
	\begin{equation}
	\label{eq:admissible-rvf}
		\left(
			\frac{\partial f^i_{\tb}}{\partial x^j}f^j_{\ta}
		\right)(p) 
		= f_{\tab}^i(p) + f_{\tc}^i(p)
	\end{equation}
for every $i\in \llbracket 1,m \rrbracket$ and every $\alpha,\beta\in \llbracket 1,n \rrbracket$. 
\end{definition}

\begin{remark}
\label{rem:admissibility_is_intrinsic}
Condition~\eqref{eq:admissible-rvf} is an intrinsic condition: Let $(U,x)$ and $(V,y)$ be charts on manifold $M$ and let $(f^i_\tau)$ and $(g^i_\tau)$ be the coordinate representations of a rough velocity field $M\to R_n^2M$ in the induced charts $(\overline{U},\overline{x})$ and $(\overline{V},\overline{y})$ on $R_n^2$, respectively. Then on $U\cap V$, we have 
	\begin{align*}
  		g^{i}_{\tab} 
  			= \frac{\partial y^{i}}{\partial x^{j}}f^{j}_{\tab} 
 		 		+ \frac{\partial^{2} y^{i}}{\partial x^{j}\partial x^{k}}f^{j}_{\ta}f^{k}_{\tb}, 
 		   \quad
 		 g^{i}_{\tc}
 			 = \frac{\partial y^{i}}{\partial x^{j}}f^{j}_{\tc}, 
 		 	\quad\mbox{and}\quad
 		 \frac{\partial g^{i}_{\tb}}{\partial y^{j}}g^{j}_{\ta}
 		 	 = f^{k}_{\ta}
				\left(
	 				\frac{\partial^{2} y^{i}}{\partial x^{k}\partial x^{p}}f^{p}_{\tb} 
	 					+ \frac{\partial y^{i}}{\partial x^{p}}\frac{ \partial f^{p}_{\tb}}{\partial x^{k}}
   				\right),
	\end{align*}
so that
	\[
 		\left(
 			\frac{\partial g^{i}_{\tb}}{\partial y^{j}}g^{j}_{\ta}
 				- g^{i}_{\tab}
  				- g^{i}_{\tc}
  		\right)
 		= \frac{\partial y^{i}}{\partial x^{j}}
 			 \left(
 				\frac{\partial f^{j}_{\tb}}{\partial x^{k}}f^{k}_{\ta}
  				- f^{j}_{\tab}
  				- f^{j}_{\tc}
  			\right).
	\]
holds for every $i\in \llbracket 1,m \rrbracket$ and every $\alpha,\beta\in \llbracket 1,n \rrbracket$.
\end{remark}

\begin{remark}
The rough velocity fields constructed in \autoref{ex:connection-rvf} and \autoref{ex:canonical-rvf} are admissible.
\end{remark}

To conclude this section we note that if we rescale an admissible rough velocity field by a smooth function, the resulting rough velocity field is also admissible. 

\begin{remark}
Let $m,n\in\Nsf$, let $M$ be an $m$-dimensional manifold, and let $(U,x)$ be a chart on $M$. Let $f: M\to R_n^2M$ be a rough velocity field whose representation in the induced chart $(\overline{U},\overline{x})$ on $R_n^2M$ is $(f^i_\tau)$. Let also $\delta: M\to \Rsf$ be a smooth function and let $\widetilde{f}:M\to R_n^2M$ be the rescaling of $f$ by $\delta$ whose representation $(\widetilde{f}^i_{\tau})$ in the chart $(\overline{U},\overline{x})$ is given by~\eqref{eq:rescaling}. Then we have
	\begin{multline*}
      \frac{\partial \widetilde{f}^{i}_{\tb}}{\partial y^{j}}\widetilde{f}^{j}_{\ta}
      	- \widetilde{f}^{i}_{\tab}
      	- \widetilde{f}^{i}_{\tc}
       =
		\frac{\partial \delta}{\partial y^{j}}f^{i}_{\tb}\delta f^{j}_{\ta} + \delta \frac{\partial f^{i}_{\tb}}{\partial y^{j}}\delta f^{j}_{\ta}
      	- \delta^{2} f^{i}_{\tab}
    	- \frac{\partial \delta}{\partial y^{j}}f^{j}_{\ta}\delta f^{i}_{\tb} - \delta^{2}f^{i}_{\tc} 
      \\
       = 
      \delta^{2}
      \left(
      	\frac{\partial f^{i}_{\tb}}{\partial y^{j}}f^{j}_{\ta}
      	- f^{i}_{\tab}
     	 - f^{i}_{\tc}
      \right)
	\end{multline*}
on $U$ for every $i\in \llbracket 1,m \rrbracket$ and every $\alpha,\beta\in \llbracket 1,n \rrbracket$. . It follows that if $f$ is admissible, the right-hand side of the above expression is zero so that $\widetilde{f}$ is also admissible.
\end{remark}

In the rest of the article, we only consider admissible rough velocity fields.

\subsection{Existence in \texorpdfstring{$\Rsf^m$}{Rᵐ}}
\label{sec:R-existence}

In this section, we prove the following

\begin{theorem}
\label{thm:R-existence}
Let $[S,T]\subseteq\Rsf$ be an interval and let $m,n\in\Nsf$. Let $p\in [1,3)$, let $\omega$ be a control function on interval $[S,T]$, and let $X$ be an $n$-dimensional branched rough path on interval $[S,T]$ with regularity $p$ and control $\omega$. Let $f: \Rsf^m\to R_n^2\Rsf^m$ be a compactly supported admissible rough velocity field and $(f^i_\tau)$ its global coordinate representation. Then for each $Y_0\in\Rsf^m$, there exists a function $Y: [S,T]\to \Rsf^m$ satisfying $Y(S)=Y_0$ for which there exists a constant $C\in (0,\infty)$ such that the inequality 
	\begin{equation}
	\label{eq:integral-curve-Rm}
		\left|
			Y^i(t) - Y^i(s) 
			- \sum_{\tau\in\Fcal_n^2} \frac{1}{\sigma(\tau)} f_{\tau}^i(Y(s)) X^\tau(s,t)
		\right|
		\leq C \omega(s,t)^\frac{3}{p}
	\end{equation}
holds for every $i\in \llbracket 1,m\rrbracket$ and $(s,t)\in \Delta[S,T]$.
\end{theorem}

As discussed in \autoref{rem:invariance_of_davie_solutions}, \autoref{thm:R-existence} cannot be proved by Picard's iteration of rough integrals with respect to the rough path $X$. Instead, we shall use Euler's approximations in the spirit of \cite{Dav08}.

Let $m,n\in \Nsf$ and let $[S,T]\subseteq\Rsf$ be an interval. Let also $p\in [1,3)$, let $\omega$ be a control function on interval $[S,T]$, and let $X$ be an $n$-dimensional branched rough path on interval $[S,T]$ with regularity $p$ and control $\omega$. Let $f: \Rsf^m\to R_n^2\Rsf^m$ be a compactly supported admissible rough velocity field and denote its coordinate  representation in the global chart by $(f_\tau^i)$ and its support by $\mathrm{supp}(f)$. Finally, let $Y_0\in \Rsf^m$. 

For a partition $\Pcal:= \{S=U_0<U_1<\ldots<U_N=T\}$ of interval $[S,T]$, we shall denote 
	\[
		X_{s,t}^\tau := X^\tau(U_s,U_t)
	\]
for $\tau\in \Fcal_n^2$ and $(s,t)\in \Delta\llbracket 0, N\rrbracket $ and
	\[
		\omega_{s,t}:= \omega(U_s,U_t)
	\]
for $(s,t)\in \Delta\llbracket 0, N\rrbracket $. We also define
	\begin{equation}
	\label{eq:euler-approx}
		Y_{t+1}^i = Y_t^i + \sum_{\tau\in\Fcal_n^2} \frac{1}{\sigma(\tau)} f_\tau^i(Y_t) X_{t,t+1}^\tau
	\end{equation}
for $i\in\llbracket 1, m\rrbracket$ and $t\in \llbracket 0,N-1\rrbracket$, and denote
	\begin{equation}
	\label{eq:euler-approx-taylor}
		Z^i_{s,t} 
			:= Y_t^i - Y^i_s 
				- \sum_{\tau\in\Fcal_n^2} \frac{1}{\sigma(\tau)} f_{\tau}^i(Y_s) X_{s,t}^\tau
	\end{equation}
for $i\in\llbracket 1, m\rrbracket$ and $(s,t)\in \Delta\llbracket 0,N\rrbracket$. 

We first prove that Euler's approximations \eqref{eq:euler-approx} stay in a ball near the initial condition. 

\begin{lemma}
\label{thm:recurrence-bound}
There exists $R\in (0,\infty)$ such that for every partition $\Pcal=\{S=U_0<U_1<\ldots<U_N=T\}$ (with $N\in\Nsf$) of the interval $[S,T]$, the estimate 
	\[
		|Y_t-Y_0| < R
	\]
holds for every $t\in \llbracket 0, N\rrbracket$. Constant $R$ depends only on $\|X\|$, $\omega(S,T)$, $p$, $\mathrm{diam}\,\mathrm{supp}(f)$, and $\|f\|_0$. 
\end{lemma}

\begin{proof}
Let $R\in (0,\infty)$ satisfy 
	\begin{equation}
	\label{eq:R-constant}
		R > \|X\| \|f\|_0 
			\left(
				\omega(S,T)^\frac{1}{p} 
				+ \omega(S,T)^\frac{2}{p}
			\right) 
			+ \mathrm{diam}\, \mathrm{supp}(f)
	\end{equation}
and let $\{S=U_0<U_1<\ldots<U_N=T\}$ (with $N\in\Nsf$) be a partition of interval $[S,T]$.

If $Y_0\not\in \mathrm{supp}(f)$, then we have $Y_t=Y_0$ for all $t\in \llbracket0, N\rrbracket$ so that we have $|Y_t-Y_0| = 0 < R$ for every $t\in \llbracket 0, N\rrbracket$. Suppose therefore that $Y_0\in \mathrm{supp}(f)$. We prove that the estimate $|Y_t - Y_0| < R$ holds for every $t\in \llbracket 0, N\rrbracket$ by induction. Clearly, the claim holds for $t=0$. Assume that $|Y_t-Y_0| < R$ holds for some $t\in \llbracket 0, N-1\rrbracket$. If $Y_t\not\in \mathrm{supp}(f)$, then $Y_{t+1}=Y_t$ so that 
	\[
		|Y_{t+1}-Y_0| = |Y_t-Y_0| < R.
	\]
If $Y_t\in \mathrm{supp}(f)$, then 
	\begin{align*}
		|Y_{t+1}-Y_0| 
			  & \leq 
			 	\left|
			 		\sum_{\tau\in\Fcal_n^2} \frac{1}{\sigma(\tau)} f_\tau (Y_t) X_{t,t+1}^\tau
			 	\right| 
			 	+ |Y_t-Y_0| 
			 \\ 
			  & \leq \|X\| \|f\|_0 
			 	\left(
			 		\omega(S,T)^\frac{1}{p} 
			 		+ \omega(S,T)^\frac{2}{p}
			 	\right) 
			 	+ \mathrm{diam}\, \mathrm{supp}(f) 
			 < R
	\end{align*}
by using assumption~\eqref{eq:rp_regularity}. This proves the induction step and concludes the proof.
\end{proof}

We continue with two general lemmas that will allow us to obtain local control of the Taylor approximation \eqref{eq:euler-approx-taylor}. As preparation, we give the following technical lemma that is easily proved by using the integral form of the remainder in Taylor's formula; see, e.g., \cite[8.14.3]{Die69}.

\begin{lemma}
\label{thm:diff-on-square-est}
Let $f:\Rsf^m\to\Rsf$ be a smooth, compactly supported function and let $\varepsilon, \tilde{\varepsilon}, \delta, \zeta\in [0,\infty)$ and $a,b, \tilde{a},\tilde{b}\in\Rsf^m$ be such that 
	\[
		|b-a|\leq \varepsilon, 
			\quad 
		|\tilde{b}-\tilde{a}|\leq \tilde{\varepsilon}, 
			\quad 
		|\tilde{a}-a|\leq \delta, 
			\quad\mbox{and}\quad 
		|\tilde{b}-\tilde{a} - (b-a)| \leq \zeta.
	\]
Let also $T_0,T_1: \Rsf^m\times\Rsf^m\to\Rsf$ be defined by 
	\begin{align*}
		& T_0(x,y):= f(y)-f(x),\\
		& T_1(x,y):= f(y)-f(x)- \frac{\partial f}{\partial x^j}(x)(y^j-x^j)
	\end{align*}
for $x,y\in\Rsf^m$. Then
	\[
		\left|T_0(a,b)\right|  \lesssim \varepsilon 
		\quad\mbox{and}\quad
		|T_0(\tilde{a},\tilde{b})-T_0(a,b)| \lesssim \varepsilon(\delta + \zeta) + \zeta,
	\]
and
	\[
		 |T_1(a,b)| \lesssim \varepsilon^2
		 \quad \mbox{and}\quad 
		 |T_1(\tilde{a},\tilde{b})-T_1(a,b)| \lesssim \varepsilon^2(\delta + \zeta) + (\varepsilon + \tilde{\varepsilon})\zeta
	\]
where the constants are non-decreasing functions of $C^1$, $C^2$, $C^2$, and $C^3$ norm of $f$, respectively.
\end{lemma}

The next lemma is used to quantify the local non-additivity of the approximations \eqref{eq:euler-approx-taylor}.

\begin{lemma}
\label{thm:dJ-estimate}
Let $\Pcal$ be a non-empty subset of the interval $[S,T]$ and let $y: \Pcal\to\Rsf^m$ be a function. Let $K:\Delta \Pcal\to\Rsf^m$ be defined by 
	\[
		K^i(s,t) := y^i(t) - y^i(s) - \sum_{\tau\in\Fcal_n^2} \frac{1}{\sigma(\tau)} f^i_\tau(y(s)) X^\tau(s,t)
	\]
for $i\in\llbracket 1, m\rrbracket$ and $(s,t)\in\Delta \Pcal$. Then the estimate
		\[
			 |K(s,t) - K(s,u) - K(u,t)| 
			 	\lesssim 
			 |K(s,u)|^2\omega(u,t)^\frac{1}{p}
			 + |K(s,u)|\omega(u,t)^\frac{1}{p}
			 + \omega(s,t)^\frac{3}{p}
		\]
holds for every $(s,u,t)\in \Pcal^3$ such that $s\leq u\leq t$ and such that $\omega(s,t)\leq 1$. The constant in the above estimate depends only on $m$, $\|f\|_2$, and $\|X\|$.
\end{lemma}

\begin{proof}
Let $i\in\llbracket 1, m\rrbracket$ and let $(s,u,t)\in \Pcal^3$ satisfy $s\leq u\leq t$ and $\omega(s,t)\leq 1$. By using the algebraic identities \eqref{eq:rp-additive}---\eqref{eq:rp-chen} together with the admissibility of the rough velocity field $f$, we obtain the decomposition
	\begin{equation}
	\label{eq:K-split}
		K^i(s,t) - K^i(s,u) - K^i(u,t) = U^i_{\tb}(s,u) X^{\tb}(u,t) + \frac{1}{2}V^i_{\tab}(s,u)X^{\tab}(u,t) + W^i_{\tc}(s,u)X^{\tc}(u,t)
	\end{equation}
where we define
	\[
		U^i_{\tb}(s,u) := U^i_{1,\tb}(s,u) + U^i_{2,\tb}(s,u)
	\]
with
	\begin{align*}
		U^i_{1,\tb}(s,u) & 
			:= f^i_{\tb}(y(u)) - f^i_{\tb}(y(s)) - \frac{\partial f^i_{\tb}}{\partial y^j} (y(s)) (y^j(u) - y^j(s)),\\
		U^i_{2,\tb}(s,u) & 
			:= \frac{\partial f^i_{\tb}}{\partial y^j}(y(s))
				\left(
					y^j(u) - y^j(s) - f^j_{\ta}(y(s)) X^{\ta}(s,u)
				\right)
	\end{align*}
and 
	\begin{align*}
		V^i_{\tab}(s,u) & := f^i_{\tab}(y(u)) - f^i_{\tab}(y(s)),\\
		W^i_{\tc}(s,u) & := f^i_{\tc}(y(u)) - f^i_{\tc}(y(s))
	\end{align*}
for $\alpha,\beta\in\llbracket 1,n\rrbracket$. By using \autoref{thm:diff-on-square-est} together with regularity estimates \eqref{eq:rp_regularity}, we obtain
	\begin{equation}
	\label{eq:K-est-1-pf}
			 |K^i(s,t) - K^i(s,u) - K^i(u,t)| 
			 \lesssim \left(
						|I(s,u)|^2 + |J(s,u)|
					  \right)
					  \omega(u,t)^\frac{1}{p}
					  + 
					  |I(s,u)|\omega(u,t)^\frac{2}{p}
	\end{equation}
where we define
	\begin{align*}
		J^i(s,u) & := y^i(u) - y^i(s) - f^i_{\ta}(y(s))X^{\ta}(s,u),\\
		I^i(s,u) & := y^i(u) - y^i(s).
	\end{align*}
By the triangle inequality, regularity estimates \eqref{eq:rp_regularity}, and the control estimate \eqref{eq:control}, we obtain
	\begin{align*}
		|J^i(s,u)| & \lesssim |K^i(s,u)| + \omega(s,u)^\frac{2}{p},\\
		|I^i(s,u)| & \lesssim |J^i(s,u)| + \omega(s,u)^\frac{1}{p} \lesssim |K^i(s,u)| + \omega(s,u)^\frac{2}{p} + \omega(s,u)^\frac{1}{p} \leq |K^i(s,u)| + \omega(s,u)^\frac{1}{p}
	\end{align*}
which, when inserted into \eqref{eq:K-est-1-pf}, yield the estimate
	\begin{align*}
	\label{eq:K-est-2-pf}
			|K^i(s,t) - K^i(s,u) - K^i(u,t)| 
			& \lesssim |K(s,u)|^2\omega(u,t)^\frac{1}{p}\\
			& \quad 	+ |K(s,u)| (\omega(u,t)^\frac{1}{p} + \omega(u,t)^\frac{2}{p})\\
			& \quad 	+ \omega(s,u)^\frac{2}{p} \omega(u,t)^\frac{1}{p}
						+ \omega(s,u)^\frac{1}{p} \omega(u,t)^\frac{2}{p}.
		\end{align*}
The claim follows by appealing to the control estimate \eqref{eq:control} again.
\end{proof}

\begin{lemma}[Discrete local sewing]
\label{thm:discrete-sewing}
Let $\theta\in (0,1)$ and let $\Pcal=\{S = U_0<U_1<\ldots<U_N = T\}$ (with $N\in\Nsf$, $N\geq 2$) be a partition of the interval $[S,T]$ satisfying the following condition:
	\begin{enumerate}[label=($P_\theta$)]
	\item\label{ass:sewing-p}
		Whenever $s,t\in\llbracket 0,N\rrbracket$ satisfy $\llbracket s+1,t-1\rrbracket\neq \emptyset$, there exists $u\in\llbracket s+1,t-1\rrbracket$ such that 
            \[
                \omega_{s,u}\leq \theta \omega_{s,t}.
            \]
	\end{enumerate}
Let $K:\Delta\Pcal\to\Rsf^m$ be a function and denote 
	\[
		K_{s,t}:=K(U_s,U_t), \quad (s,t)\in\Delta\llbracket 0,N\rrbracket.
	\]
Assume that the following condition is satisfied:
	\begin{enumerate}[label=($K$)]
	\item\label{ass:sewing-K}
		It holds that $K_{s,s}=0$ for all $s\in\llbracket 0,N\rrbracket$ and $K_{s,s+1}=0$ for all $s\in \llbracket 0,N-1\rrbracket$. Additionally, there exist $C,M\in (0,\infty)$ such that the estimate
	\[
		| K_{s,t} - K_{s,u} - K_{u,t}| \leq C |K_{s,u}|\,\omega_{u,t}^\frac{1}{p} + M\omega_{s,t}^\frac{3}{p}
	\]
holds for every $(s,u,t)\in\llbracket 0,N\rrbracket^3$ satisfying $s\leq u\leq t$ and $\omega_{s,t}\leq 1$.
	\end{enumerate}
Then there exist $\delta \in (0,1)$ and $L\in (0,\infty)$ such that the estimate
	\[
		|K_{s,t}| \leq L M \omega_{s,t}^\frac{3}{p}
	\]
holds for every $s,t\in\Delta\llbracket 0,N\rrbracket$ satisfying $\omega_{s,t}<\delta$. Constant $\delta$ depends only on $p$, $\theta$, and $C$ while constant $L$ depends only on $\theta$ and $p$. 
\end{lemma}

\begin{proof}
Since $3/p>1$, we have that $1- \theta^\frac{3}{p} - (1-\theta)^\frac{3}{p}>0$ and therefore, we can choose $\delta\in (0,1)$ and $L\in (0,\infty)$ that satisfy the following inequalities:
	\begin{equation}
	\label{eq:delta-L}
		2C\delta^\frac{1}{p} \leq \frac{1}{2}\left(1-\theta^\frac{3}{p}-(1-\theta)^\frac{3}{p}\right)
		\quad\mbox{and}\quad 
		L\geq  \frac{4}{1-\theta^\frac{3}{p} - (1-\theta)^\frac{3}{p}}.
	\end{equation}
We proceed by induction. Note first that $|J_{s,s}|=0$ for every $s\in \llbracket 0, N\rrbracket$ and $|J_{s,s+1}|=0$ for every $s\in \llbracket 0,N-1\rrbracket$ by assumption \ref{ass:sewing-K}. Let $v\in \llbracket 2,N\rrbracket$ and suppose that
	\[
		|K_{s,t}| \leq LM\omega_{s,t}^\frac{3}{p}
	\]
holds for every $s,t\in \llbracket 0, N-v\rrbracket$ satisfying $s\leq t<s+v$ and $\omega_{s,t}<\delta$. Let $s\in \llbracket 0,N-v\rrbracket$ and let $t=s+v$ and assume that $\omega_{s,t}<\delta$. Since $\Pcal$ satisfies \ref{ass:sewing-p}, there exists $u\in \llbracket s+1,t-1\rrbracket$ such that 
	\begin{equation}
	\label{eq:theta-1}
		\omega_{s,u}\leq \theta \omega_{s,t}
	\end{equation}
and we choose $u$ to be the largest such index. Then we have that $\omega_{s,u+1}>\theta\omega_{s,t}$ and, consequently, 
	\begin{equation}
	\label{eq:theta-2}
		\omega_{u+1,t}\leq \omega_{s,t} - \omega_{s,u+1} < (1-\theta)\omega_{s,t}
	\end{equation}
by superadditivity of the control function. Then we obtain
	\begin{align*}
		|K_{s,t}| & \leq |K_{s,t} - K_{s,u} - K_{u,t}| + |K_{s,u}| + |K_{u,t}-K_{u,u+1} - K_{u+1,t}| + |K_{u, u+1}| + |K_{u+1,t}| \\
			& \leq C|K_{s,u}|\omega_{u,t}^\frac{1}{p} + M\omega_{s,t}^\frac{3}{p} + |K_{s,u}| + C|K_{u,t}| \omega_{u+1,t}^\frac{1}{p} + M\omega_{u,t}^\frac{3}{p} + |K_{u+1,t}|\\
			& \leq C LM\omega_{s,u}^\frac{3}{p}\omega_{u,t}^\frac{1}{p} + M\omega_{s,t}^\frac{3}{p} + LM\omega_{s,u}^\frac{3}{p} + CLM\omega_{u,t}^\frac{3}{p} + M\omega_{u,t}^\frac{3}{p} + LM\omega_{u+1,t}^\frac{3}{p} \\
			& \leq (2 + (\theta^\frac{3}{p} + (1-\theta)^\frac{3}{p} + 2C\delta^\frac{1}{p})L) M\, \omega_{s,t}^\frac{3}{p}
	\end{align*}
by using assumption \ref{ass:sewing-K}, the induction hypothesis, and the estimates  \eqref{eq:theta-1} and \eqref{eq:theta-2} successively. As $\delta$ and $L$ were chosen so that the inequalities in \eqref{eq:delta-L} are satisfied, we have that
	\[
		|K_{s,t}| \leq L M\omega_{s,t}^\frac{3}{p}
	\]
holds as desired.
\end{proof}

\begin{remark}
Note that if there exists $\theta\in (0,1)$ such that a partition 
	\[
		\Pcal=\{S = U_0<U_1<\ldots<U_N = T\}\quad (\mbox{with } N\in\Nsf, N\geq 2)
	\]
satisfies
	\begin{equation}
	\label{ass:sewing-p2}
		\max_{s\in \llbracket 0, N-1\rrbracket} \omega_{s,s+1} \leq \theta \min_{\substack{s,t\in \llbracket 0,N\rrbracket \\ t> s+1}} \omega_{s,t},
	\end{equation}
then $\Pcal$ also satisfies \ref{ass:sewing-p}. If, for example, the set $\Pcal$ is uniform in $\omega$, i.e.\ there exists $\varepsilon\in (0,\infty)$ satisfying $\omega_{s,s+1} = \varepsilon$ for every $s\in \llbracket 0,N-1 \rrbracket$, then $\Pcal$ satisfies the above condition \eqref{ass:sewing-p2}, and therefore also condition \ref{ass:sewing-p}, with any $\theta\in [1/2,1)$. 
Let us also note that there exists a sequence of partitions $\{\Pcal^{(a)}\}_{a\in\Nsf}$, denoted by 
	\[
		\Pcal^{(a)}:=\{S=U_0^a<U_1^a<\ldots<U_{N_a}^a=T\} \quad (\mbox{with } N_a\in \Nsf, N_a\geq 2),
	\]
such that the following holds:
	\begin{enumerate}[label=(\alph*)]
    \itemsep0em
	\item There exists $\theta\in (0,1)$ such that for every $a\in\Nsf$, partition $\Pcal^{(a)}$ satisfies \ref{ass:sewing-p}.
	\item For every $a\in\Nsf$, partition $P^{(a+1)}$ is a refinement of $\Pcal^{(a)}$. 
	\item There is the convergence $|\Pcal^{(a)}|:=\max_{i\in\llbracket 0,N_a-1\rrbracket } |U_{i+1}^a-U^a_i|\to 0$ as $a\to\infty$.
	\end{enumerate}
Indeed, consider the sequence $\{D_n\}_{n\in\Nsf}$ of partitions $D_n:=\{S=U_0^n<U_1^n<\ldots <U_{2^n}^n=T\}$ such that $\omega_{s,s+1}=2^{-n}\omega(S,T)$ for every $s\in \llbracket 0, 2^{n}-1\rrbracket$ holds. Then for every $\theta\in [1/2,1)$ we have that for every $n\in\Nsf$, $D_n$ satisfies \ref{ass:sewing-p}, $D_{n}\subseteq D_{n+1}$, and $|D_n|\to 0$ as $n\to\infty$ by continuity of $\omega$. 
\end{remark}

We now apply the two results above to obtain a local bound on the approximations \eqref{eq:euler-approx-taylor}.

\begin{lemma}
\label{thm:J-estimate}
Let $\theta\in (0,1)$. There exist $\delta\in (0,1)$ and $L\in (0,\infty)$ such that for every partition $\Pcal=\{S=U_0<U_1<\ldots<U_N=T\}$ (with $N\in\Nsf$, $N\geq 2$) satisfying \ref{ass:sewing-p} the estimate
	\[
		|Z_{s,t}|\leq L\omega_{s,t}^\frac{3}{p}
	\]
holds for every $(s,t)\in \Delta\llbracket 0,N \rrbracket$ satisfying $\omega_{s,t}<\delta$. Both constants $\delta$ and $L$ depend only on $\theta$, $m$, $\|f\|_2$, $\mathrm{supp}(f)$, $p$, $\omega(S,T)$, and $\|X\|$. 
\end{lemma}

\begin{proof}
Let $\Pcal=\{S=U_0<U_1<\ldots<U_N=T\}$ (with $N\in\Nsf$, $N\geq 2$) be a partition of interval $[S,T]$ that satisfies \ref{ass:sewing-p}. We have that $Z_{s,s}=0$ for every $s\in \llbracket 0,N\rrbracket$ by definition and we also have $Z_{s,s+1}=0$ for every $s\in \llbracket 0,N-1\rrbracket$ by definition of $Y$. Moreover, by \autoref{thm:dJ-estimate}, we have for every $(s,u,t)\in \llbracket 0, N\rrbracket^3$ satisfying $s\leq u\leq t$ and $\omega_{s,t}\leq 1$ the estimate 
	\[ 
		|Z_{s,t}-Z_{s,u}-Z_{u,t}| \lesssim |Z_{s,u}|^2 \omega_{u,t}^\frac{1}{p} + |Z_{s,u}|\omega_{u,t}^\frac{1}{p} + \omega_{s,t}^\frac{3}{p},
	\]
but, by \autoref{thm:recurrence-bound}, we have that there exists $R\in (0,\infty)$ such that
	\[ 
		|Z_{s,u}| \leq |Y_u-Y_0| + |Y_s-Y_0| + \left|\sum_{\tau\in\Fcal_n^2} \frac{1}{\sigma(\tau)} f_{\tau}(Y_s)X_{s,u}^\tau\right| < 3R
	\]
holds (see \eqref{eq:R-constant}), so that in fact, we have the estimate
	\[
		|Z_{s,t}-Z_{s,u}-Z_{u,t}| \lesssim |Z_{s,u}|\,\omega_{u,t}^\frac{1}{p} + \omega_{s,t}^\frac{3}{p}.
	\]
Thus, condition \ref{ass:sewing-K} of \autoref{thm:discrete-sewing} is verified and we can infer that there exists $\delta\in (0,1)$ and $L\in (0,\infty)$, that both depend only on $\theta$, $m$, $\|X\|$, $\omega(S,T)$, $p$, $\|f\|_2$, and $\mathrm{supp}(f)$, such that 
	\[
		(s,t)\in\Delta\llbracket 0,N\rrbracket: \omega_{s,t}<\delta \quad\implies\quad |Z_{s,t}|\leq L\omega_{s,t}^\frac{3}{p}
	\]
as desired.
\end{proof}

Combining the local bound obtained in \autoref{thm:J-estimate} with the global boundedness of Euler's approximations from \autoref{thm:recurrence-bound} yields a global bound on Taylor's approximations \eqref{eq:euler-approx-taylor}.

\begin{lemma}
\label{thm:J-estimate-2}
Let $\theta\in (0,1)$. There exists $C\in (0,\infty)$ such that for every partition $\Pcal=\{S=U_0<U_1<\ldots<U_N=T\}$ (with $N\in\Nsf$, $N\geq 2$) of the interval $[S,T]$ satisfying \ref{ass:sewing-p}, we have
	\[
		|Z_{s,t}| \leq C\omega_{s,t}^\frac{3}{p}
	\]
for every $(s,t)\in \Delta \llbracket 0, N\rrbracket$. The constant $C$ depends only on $\theta$, $m$, $\|f\|_2$, $\mathrm{supp}(f)$, $p$, $\omega(S,T)$, and $\|X\|$.
\end{lemma}

\begin{proof}
Let $R$ be as in \autoref{thm:recurrence-bound} and $\delta$ and $L$ be as in \autoref{thm:J-estimate}. Let $\Pcal:=\{S=U_0<U_1<\ldots<U_N=T\}$ (with $N\in\Nsf$, $N\geq 2$) be a partition of interval $[S,T]$ satisfying \ref{ass:sewing-p}. Let $C\in (0,\infty)$ be such that 
	\[
		C \geq L + 3R\delta^{-\frac{3}{p}}.
	\]
Let $s,t\in \Delta\llbracket 0,N\rrbracket$. If $\omega_{s,t}<\delta$, then $|Z_{s,t}|\leq L\omega_{s,t}^\frac{3}{p} < C\omega_{s,t}^\frac{3}{p}$ by \autoref{thm:J-estimate}. If $\omega_{s,t}\geq \delta$, then 
	\[
		|Z_{s,t}| 
			\leq 
				|Y_t-Y_0| 
				+ |Y_s-Y_0| 
				+ \left|
					\sum_{\tau\in\Fcal_n^2} \frac{1}{\sigma(\tau)} f_\tau(Y_s) X_{s,t}^\tau
				\right| 			
			< 
				3R
			< 
				C \delta^\frac{3}{p} 
			\leq 
				C\omega_{s,t}^\frac{3}{p}
	\]
by using \eqref{eq:R-constant}.
\end{proof}

Finally, we prove \autoref{thm:R-existence} by constructing the integral curve of the rough velocity field $f$ driven by the rough path $X$ as a limit of Euler approximations \eqref{eq:euler-approx} along a sequence of partitions of the interval $[S,T]$ whose mesh converges to zero.

\begin{proof}[Proof of \autoref{thm:R-existence}]

Let $\{\Pcal^{(a)}\}_{a\in\Nsf}$ be a sequence of partitions of interval $[S,T]$ denoted by 
	\[
		\Pcal^{(a)} = \{S=U_0^a<U_1^a <\ldots <U_{N_a}^a = T\}  \quad (\mbox{with } N_a\in\Nsf, N_a\geq 2),
	\]
such that the following hold:
	\begin{enumerate}[label=(\alph*)]
	\item There exists $\theta\in (0,1)$ such that for every $a\in\Nsf$, partition $\Pcal^{(a)}$ satisfies \ref{ass:sewing-p}.
	\item\label{refinement} For every $a\in\Nsf$, $\Pcal^{(a+1)}$ is a refinement of $\Pcal^{(a)}$. 
	\item\label{mesh} There is the convergence $|\Pcal^{(a)}|\to 0$ as $a\to\infty$.
	\end{enumerate}
For $a\in\Nsf$, let $\{Y_k^a\}_{k\in \llbracket 0,N^a\rrbracket}$ be the sequence defined by~\eqref{eq:euler-approx} corresponding to partition $\Pcal^{(a)}$ starting at $Y_0$. Let $\Scal:= \bigcup_{a\in\Nsf} \Pcal^{(a)}$. By property \ref{refinement} above, for every $s\in\Scal$, there exists $a_0(s)\in \Nsf$ such that for every $a\geq a_0(s)$ there is a unique $k$ such that $U_k^a=s$ and we define $Y^a(s) := Y_k^a$. Then there are constants $R\in (0,\infty)$ and $C\in (0,\infty)$ such that the estimate
	\begin{equation}
	\label{eq:bddness-Ya}
		|Y_0 - Y^{a}(s)| < R
	\end{equation}
holds for all $s\in\Scal$ and $a\geq a_0(s)$ and the estimate
	\begin{equation}
	\label{eq:intcurve-Ya-S}
		\left| 
			Y^{a}(t) - Y^{a}(s) 
			- \sum_{\tau\in\Fcal_n^2} \frac{1}{\sigma(\tau)} f_\tau(Y^a(s))X^\tau(s,t) 
		\right|
		\leq 
		C\omega(s,t)^\frac{3}{p}
	\end{equation}
holds for all $(s,t)\in\Delta\Scal$ and $a\geq \max\{a_0(s),a_0(t)\}$ by \autoref{thm:recurrence-bound} and \autoref{thm:J-estimate-2}, respectively. It follows from~\eqref{eq:bddness-Ya} by a diagonal argument that there is a subsequence $\{a_\ell\}_\ell$, $a_\ell\xrightarrow[\ell\to\infty]{}\infty$, such that for all $s\in\Scal$, $Y^{a_\ell}(s)$ converges to some $Y(s)\in\Rsf^m$ as $\ell\to\infty$. It follows from~\eqref{eq:intcurve-Ya-S} that there is the estimate
	\[ 
		\left|
			Y(t)-Y(s) 
			- \sum_{\tau\in\Fcal_n^2} \frac{1}{\sigma(\tau)} f_{\tau}(Y(s)) X^\tau(s,t)
		\right| 
		\leq 
		C\omega(s,t)^\frac{3}{p}
	\]
for all $(s,t)\in\Delta\Scal$ so that
	\[
		|Y(t)-Y(s)| 
			\leq 
			\|f_0\|\|X\| \left(\omega(s,t)^\frac{1}{p} 
			+ \omega(s,t)^\frac{2}{p} \right)
			+ C\omega(s,t)^\frac{3}{p}
	\]
holds for every $(s,t)\in\Delta\Scal$. Thus it is seen that $Y:\Scal\to \Rsf^m$ has a unique continuous extension to the closure of $\Scal$. By \ref{mesh}, this closure is the whole interval $[S,T]$, and, by continuity, this extension satisfies~\eqref{eq:integral-curve-Rm} as desired.
\end{proof}

\subsection{Uniqueness in \texorpdfstring{$\Rsf^m$}{Rᵐ}}
\label{sec:R-uniqueness}

In this section, we will prove the following

\begin{theorem}
\label{thm:R-uniqueness}
Let $[S,T]\subseteq\Rsf$ be an interval and let $m,n\in\Nsf$. Let $p\in [1,3)$, let $\omega$ be a control function on the interval $[S,T]$, and let $X$ be an $n$-dimensional branched rough path on the interval $[S,T]$ with regularity $p$ and control $\omega$. Let $f:\Rsf^m\to R_n^2\Rsf^m$ be a compactly supported admissible rough velocity field. Finally, let $Y, \widetilde{Y}: [S,T]\to \Rsf^m$ be two integral curves of $f$ driven by $X$ such that $Y(S)=\widetilde{Y}(S)$. Then there exists $\varepsilon\in (0,\infty)$ such that $Y=\widetilde{Y}$ on $[S,S+\varepsilon]$. 
\end{theorem}

Let $m,n\in\Nsf$ and let $[S,T]\subseteq\Rsf$ be an interval. Let also $p\in [1,3)$, let $\omega$ be a control function on interval $[S,T]$, and let $X$ be an $n$-dimensional branched rough path on interval $[S,T]$ with regularity $p$ and control $\omega$. Let $f:\Rsf^m\to R_n^2\Rsf^m$ be a compactly supported admissible rough velocity field and denote its coordinate representation in the global chart by $(f^i_\tau)$. Let $Y, \widetilde{Y}: [S,T]\to\Rsf^m$ be two integral curves of $f$ driven by $X$ such that $Y(S)=\widetilde{Y}(S)$. Denote
	\begin{align*}
		Z^i(s,t):= Y^i(t)-Y^i(s) - \sum_{\tau\in\Fcal_n^2} \frac{1}{\sigma(\tau)} f^i_\tau(Y(s))X^\tau(s,t),\\
		\widetilde{Z}^i(s,t):= \widetilde{Y}^i(t)-\widetilde{Y}^i(s) - \sum_{\tau\in\Fcal_n^2} \frac{1}{\sigma(\tau)} f^i_\tau(\widetilde{Y}(s))X^\tau(s,t),
	\end{align*}
and 
	\[ 
		\overline{Z}^i(s,t):= Z^i(s,t)-\widetilde{Z}^i(s,t)
	\]
for $i\in\llbracket 1,m \rrbracket$ and $(s,t)\in\Delta[S,T]$. As $Y$ and $\widetilde{Y}$ are integral curves, there exist finite positive constants $C, \widetilde{C}$ such that 
	\begin{equation}
	\label{eq:KK-est}
		|Z(s,t)|\leq C\omega(s,t)^\frac{3}{p}
			\quad\mbox{and}\quad 
		|\widetilde{Z}(s,t)| \leq C \omega(s,t)^\frac{3}{p}
	\end{equation}
for every $(s,t)\in\Delta[S,T]$. We give several technical lemmas first.

\begin{lemma}
\label{thm:R-uniqueness-L1}
The estimate 
	\[
		|\widetilde{Y}(t)-\widetilde{Y}(s) - (Y(t)-Y(s))| \lesssim |\overline{Z}(s,t)| + |\widetilde{Y}(s)-Y(s)|\omega(s,t)^\frac{1}{p}
	\]
holds for every $(s,t)\in\Delta[S,T]$ such that $\omega(s,t)\leq 1$. The constant in the above estimate depends only on $C$, $\widetilde{C}$, $m$, $\|f\|_1$, and $\|X\|$.
\end{lemma}

\begin{proof}
For $i\in \llbracket 1, m\rrbracket$ and $(s,t)\in\Delta[S,T]$, set
	\begin{align*}
		J^i(s,t) & := Y^i(t)-Y^i(s) - f^i_{\ta}(Y(s))X^{\ta}(s,t),\\
		\widetilde{J}^i(s,t) & := \widetilde Y^i(t)-\widetilde Y^i(s) - f^i_{\ta}(\widetilde Y(s))X^{\ta}(s,t), \\
		\overline{J}^i(s,t) & := \widetilde J^i(s,t)-J^i(s,t)
	\end{align*}
and 
	\begin{align*}
		I^i(s,t) & := Y^i(t)-Y^i(s),\\
		\widetilde I^i(s,t) & := \widetilde Y^i(t) - \widetilde Y^i(s),\\
		\overline{I}^i(s,t) & := \widetilde I^i(s,t)-I^i(s,t).
	\end{align*} 
Let $i\in\llbracket 1,m\rrbracket$ and let $(s,t)\in\Delta[S,T]$ be such that $\omega(s,t)\leq 1$. By using the Triangle Inequality, \autoref{thm:diff-on-square-est}, the regularity estimates in \eqref{eq:rp_regularity}, and the estimate \eqref{eq:control}, we obtain
	\begin{align*}
		|\overline{J}^i(s,t)| 
			 & \lesssim |\overline{Z}^i(s,t)| 
			  			 + |\widetilde{y}(s) - y(s)| \omega(s,t)^\frac{2}{p},
	\\
	    |\overline{I}^i(s,t)| 
	    	 & \lesssim |\overline{J}^i(s,t)| 
	    	 			 + |\tilde{y}(s)-y(s)|\omega(s,u)^\frac{1}{p}
	    	   \lesssim |\overline{Z}^i(s,t)| 
	    	   			 + |\tilde{y}(s)-y(s)|\omega(s,u)^\frac{1}{p}.
	\end{align*}
\end{proof}

\begin{lemma}
\label{thm:R-uniqueness-L2}
The estimate 
		\[
			 |\overline{Z}(s,t) - \overline{Z}(s,u) - \overline{Z}(u,t)| 
			 	\lesssim 
			 |\overline{Z}(s,u)|\omega(u,t)^\frac{1}{p}
			 + |\widetilde{Y}(s)-Y(s)|\omega(s,t)^\frac{3}{p}
		\]
holds for every $(s,u,t)\in [S,T]^3$ such that $s\leq u\leq t$ and such that $\omega(s,t)\leq 1$. The constant in the above estimate depends only on $C, \widetilde{C}$, $m$, $\|f\|_3$, and $\|X\|$.
\end{lemma}

\begin{proof}
Let $i\in \llbracket 1,m \rrbracket$ and let $(s,u,t)\in[S,T]^3$ be such that $s\leq u\leq t$ and $\omega(s,t)\leq 1$. Let $I,\widetilde{I}, \overline{I}$ and $J, \widetilde{J}, \overline{J}$ be as in the proof of \autoref{thm:R-uniqueness-L1}. Using assumption \eqref{eq:KK-est}, the regularity estimates in \eqref{eq:rp_regularity}, and the control estimate \eqref{eq:control} yields
	\begin{align*}
		& |J(s,u)|  \lesssim \omega(s,u)^\frac{2}{p}
		\quad\mbox{and}\quad 
		|\widetilde{J}(s,u)| \lesssim \omega(s,u)^\frac{2}{p},
		\\
		& |I(s,u)| \lesssim \omega(s,u)^\frac{1}{p}
		\quad\mbox{and}\quad 
		|\widetilde{I}(s,u)| \lesssim \omega(s,u)^\frac{1}{p}.
	\end{align*}
and, by the proof of \autoref{thm:R-uniqueness-L1}, we also have
	\begin{align*}
		|\overline{J}^i(s,u)| 
			 & \lesssim |\overline{Z}^i(s,u)| 
			  			 + |\widetilde{Y}(s) - Y(s)| \omega(s,u)^\frac{2}{p},
	\\
	    |\overline{I}^i(s,u)| 
	    	 &
	    	   \lesssim |\overline{Z}^i(s,u)| 
	    	   			 + |\tilde{Y}(s)-Y(s)|\omega(s,u)^\frac{1}{p}.
	\end{align*}
By using the decomposition \eqref{eq:K-split} for both $Z$ and $\widetilde{Z}$, we obtain
\begin{equation}
\label{eq:K-split-2}
	\overline{Z}^i(s,t) 
		- \overline{Z}^i(s,u) 
		- \overline{Z}^i(u,t) 
		= \overline{U}^i_{\tb}(s,u)X^{\tb}(u,t) 
			+ \frac{1}{2} \overline{V}_{\tab}^i(s,u)X^{\tab}(u,t) 
			+ \overline{W}^i_{\tc}(s,u)X^{\tc}(u,t)
\end{equation}
where we define
	\[
		\overline{U}_{\tb}^i(s,u) := \overline{U}_{1,\tb}^i(s,u) + \overline{U}_{2,\tb}^i(s,u) + \overline{U}_{3,\tb}^i(s,u)
	\]
with
	\begin{align*}
		\overline{U}^i_{1,\tb}(s,u) & := \left(
					f^i_{\tb}(\widetilde y(u)) 
					- f^i_{\tb}(\widetilde y(s)) 
					- \frac{\partial f^i_{\tb}}{\partial y^j}(\widetilde y(s)) \widetilde{I}^j(s,u)
				\right) \\
				& \qquad - \left(
					f^i_{\tb}(y(u)) 
					- f^i_{\tb}(y(s)) 
					- \frac{\partial f^i_{\tb}}{\partial y^j}(y(s)) I^j(s,u)
				\right),\\
		\overline{U}^i_{2,\tb}(s,u) & := 
					\left(
						\frac{\partial f^i_{\tb}}{\partial y^j} (\widetilde y(s))
						- \frac{\partial f^i_{\tb}}{\partial y^j} (y(s))
					\right)
					\widetilde{J}^j(s,u), \\
		\overline{U}^i_{3,\tb}(s,u) & := 
			\frac{\partial f^i_{\tb}}{\partial y^j}(y(s)) \overline{J}^j(s,u)
	\end{align*}
for $\beta\in \llbracket 1,n\rrbracket$ and 
	\begin{align*}
		\overline{V}^i_{\tab}(s,u) & := f^i_{\tab}(\widetilde{y}(u))-f^i_{\tab}(\widetilde{y}(s)) - (f^i_{\tab}(y(u))-f^i_{\tab}(y(s))),\\
		\overline{W}^i_{\tc}(s,u) & := f^i_{\tc}(\widetilde{y}(u))-f^i_{\tc}(\widetilde{y}(s)) - (f^i_{\tc}(y(u))-f^i_{\tc}(y(s)))
	\end{align*}
for $\alpha,\beta\in\llbracket 1,n\rrbracket$. It follows by appealing to \autoref{thm:diff-on-square-est} and to the control estimate \eqref{eq:control} that
	\begin{align*}
		|\overline{U}^i_{1,\tb}(s,u)| & \lesssim |\overline{K}(s,u)|\omega(s,u)^\frac{1}{p} + |\widetilde{y}(s)-y(s)| \omega(s,u)^\frac{2}{p}, \\
		|\overline{U}^i_{2,\tb}(s,u)| & \lesssim |\widetilde{y}(s)-y(s)|\omega(s,u)^\frac{2}{p}, \\
		|\overline{U}^i_{3,\tb}(s,u)| & \lesssim |\overline{K}(s,u)| + |\widetilde{y}(s) - y(s)|\omega(s,u)^\frac{2}{p},\\
		|\overline{V}^i_{\tab}(s,u)| & \lesssim |\overline{K}(s,u)| + |\widetilde{y}(s)-y(s)| \omega(s,u)^\frac{1}{p},\\
		|\overline{W}^i_{\tab}(s,u)| & \lesssim |\overline{K}(s,u)| + |\widetilde{y}(s)-y(s)| \omega(s,u)^\frac{1}{p}.
	\end{align*}
The claim now follows by inserting the above estimates into \eqref{eq:K-split-2} and using the regularity estimates in \eqref{eq:rp_regularity} and the control estimate \eqref{eq:control} again.
\end{proof}

We continue with an analogue of \autoref{thm:discrete-sewing} in continuous time.

\begin{lemma}[Continuous local sewing]
\label{thm:cts-sewing}
Let $[U,V]\subseteq\Rsf$ be an interval and let $K:\Delta[U,V]\to\Rsf^m$ be a function. Assume that $|K(s,t)|=0$ for every $(s,t)\in\Delta[U,V]$ such that $\omega(s,t)=0$ and that there are constants $C,M\in (0,\infty)$ such that 
	\begin{equation}
	\label{eq:cts-sewing-dK}
		|K(s,t)-K(s,u)-K(u,t)| \leq C|K(s,u)|\omega(u,t)^\frac{1}{p} + M\omega(s,t)^\frac{3}{p}
	\end{equation}
holds for every $(s,t)\in [U,V]^3$ satisfying $s\leq u\leq t$ and $\omega(s,t)\leq 1$. Then there exist constants $\delta\in (0,1)$ and $L \in (0,\infty)$ such that 
	\[
		| K(s,t) | \leq L M\omega(s,t)^\frac{3}{p}
	\] 
holds for every $(s,t)\in\Delta[U,V]$ such that $\omega(s,t)<\delta$. The constant $\delta$ depends only on $p$ and $C$ while $L$ only on $p$.
\end{lemma}

\begin{proof}
Fix any $\theta\in (0,1)$. Since $3/p>1$, we have that $1-\theta^\frac{3}{p}-(1-\theta)^\frac{3}{p}>0$ and therefore, we can choose $\delta\in (0,1)$ and any $L\in (0,\infty)$ that satisfy the following inequalities:
	\begin{equation}
	\label{eq:cts-sewing-delta-L}
		C\delta^\frac{1}{p} \leq \frac{1}{2}\left(1-\theta^\frac{3}{p} - (1-\theta)^\frac{3}{p}\right) \quad\mbox{and}\quad L\geq \frac{2}{1-\theta^\frac{3}{p} - (1-\theta)^\frac{3}{p}}.
	\end{equation}
Note first that if $(s,t)\in\Delta[U,V]$ is such that $\omega(s,t)=0$, then $|K(s,t)|=0$ holds by the assumption of the lemma and the claim holds trivially. For $(s,t)\in\Delta[U,V]$ such that $\omega(s,t) >0$, we denote
	\[ 
		u^* := \sup\{ u\in (s,t): \omega(s,u) = \theta\omega(s,t) \}.
	\]
(A point $u\in (s,t)$ that satisfies $\omega(s,u)=\theta\omega(s,t)$ always exists by the Intermediate Value Theorem.) It then follows by superadditivity of the control function that $\omega(u^*,t)\leq (1-\theta)\omega(s,t)$ but it can happen $\omega(u^*,t)=0$. We thus set
	\begin{align*}
		\Delta_0(\delta) & := \{(s,t)\in\Delta[S,T]: 0<\omega(s,t)<\delta\},\\
		\Delta_1(\delta) & := \{(s,t)\in\Delta[S,T]: 0<\omega(s,t)<\delta, \omega(u^*,t)>0\},\\
		\Delta_2(\delta) & := \{(s,t)\in\Delta[S,T]: 0<\omega(s,t)<\delta, \omega(u^*,t)=0\}
	\end{align*}
and also 
	\[
		F_j(\delta):= \sup_{(s,t)\in \Delta_j(\delta)} \frac{|K(s,t)|}{\omega(s,t)^\frac{3}{p}}, \quad j\in\{0,1,2\}.
	\]
Clearly, we have that $F_0=\max\{F_1,F_2\}$. Let now $(s,t)\in\Delta_0(\delta)$. We have that 
	\begin{equation}
	\label{eq:cts-sewing-est-split}
		|K(s,t)| \leq |K(s,u^*)| + |K(u^*,t)| + C|K(s,u^*)|\delta^\frac{1}{p} + M\omega(s,t)^\frac{3}{p}
	\end{equation}
by appealing to estimate \eqref{eq:cts-sewing-dK}. 
If $(s,t)\in\Delta_1(\delta)$, we divide inequality \eqref{eq:cts-sewing-est-split} by $\omega(s,t)^\frac{3}{p}$ and use 
	\[
		\frac{\omega(s,u^*)^\frac{3}{p}}{\omega(s,t)^\frac{3}{p}} =\theta^\frac{3}{p}, \quad \frac{\omega(u^*,t)^\frac{3}{p}}{\omega(s,t)^\frac{3}{p}} \leq (1-\theta)^\frac{3}{p}, \quad \frac{\omega(u^*,t)^\frac{3}{p}}{\omega(s,t)^\frac{3}{p}} \leq 1
	\]
to obtain 
	\[
		\frac{|K(s,t)|}{\omega(s,t)^\frac{3}{p}} \leq \left(\theta^\frac{3}{p} + (1-\theta)^\frac{3}{p} + C\delta^\frac{1}{p}\right) F_0(\delta) + M
	\]
which yields the inequality
	\begin{equation}
	\label{eq:cts-sewing-est-f1}
		F_1(\delta) \leq \left(\theta^\frac{3}{p} + (1-\theta)^\frac{3}{p} + C\delta^\frac{1}{p}\right) F_0(\delta) + M.
	\end{equation}
If, on the other hand, $(s,t)\in\Delta_2(\delta)$, we have $|K(u^*,t)|=0$ and inequality \eqref{eq:cts-sewing-est-split} reduces to 
	\[
		|K(s,t)| \leq |K(s,u^*)| + C|K(s,u^*)|\delta^\frac{1}{p} + M\omega(s,t)^\frac{3}{p}.
	\]
By the same arguments as before, we obtain 
	\begin{equation}
	\label{eq:cts-sewing-est-f2} 
		F_2(\delta) \leq \left( \theta^\frac{3}{p} + C\delta^\frac{1}{p}\right) F_0(\delta) + M.
	\end{equation}
Putting together \eqref{eq:cts-sewing-est-f1} and \eqref{eq:cts-sewing-est-f2} 
	\[
		F_0(\delta) = \max\{F_1(\delta),F_2(\delta)\} \leq \left(\theta^\frac{3}{p} + (1-\theta)^\frac{3}{p} + C\delta^\frac{1}{p}\right) F_0(\delta) + M
	\]
and as $\delta$ and $L$ were chosen so that \eqref{eq:cts-sewing-delta-L} are satisfied, we obtain $F_0(\delta)\leq LM$ as desired.
\end{proof}

\begin{lemma}
\label{thm:R-uniqueness-L3}
There exist constants $\delta\in (0,1)$ and $L\in (0,\infty)$ such that for every interval $[U,V]\subseteq [S,T]$, we have 
	\[
		|\overline{Z}(s,t)|\leq L \sup_{r\in [U,V]} |\widetilde{Y}(r)-Y(r)| \omega(s,t)^\frac{3}{p}
	\]
for every $(s,t)\in\Delta[U,V]$ satisfying $\omega(s,t)<\delta$. Both constants $\delta$ and $L$ depend only on $C, \widetilde{C}$, $m$, $\|f\|_3$, $\|X\|$, and $p$.
\end{lemma}

\begin{proof}
By \autoref{thm:R-uniqueness-L2}, we have that
	\[
			 |\overline{Z}(s,t) - \overline{Z}(s,u) - \overline{Z}(u,t)|
			 	\lesssim 
			 |\overline{Z}(s,u)|\omega(u,t)^\frac{1}{p}
			 + |\widetilde{Y}(s)-Y(s)|\omega(s,t)^\frac{3}{p}
	\]
holds for every $(s,u,t)\in[S,T]^3$ satisfying $s\leq u\leq t$ and $\omega(s,t)\leq 1$ with the constant depending only on $C, \widetilde{C}, m, \|f\|_3,$ and $\|X\|$. Thus we have, for an interval $[U,V]\subseteq [S,T]$, that 
	\[
			 |\overline{Z}(s,t) - \overline{Z}(s,u) - \overline{Z}(u,t)|
			 	\lesssim 
			 |\overline{Z}(s,u)|\omega(u,t)^\frac{1}{p}
			 + \sup_{r\in [U,V]}|\widetilde{Y}(r)-Y(r)|\omega(s,t)^\frac{3}{p}
	\]
holds whenever $(s,u,t)\in [U,V]^3$ satisfies $s\leq u\leq t$ and $\omega(s,t)\leq 1$. The claim now follows by \autoref{thm:cts-sewing}.
\end{proof}

The proof of \autoref{thm:R-uniqueness} can be given now.

\begin{proof}[Proof of \autoref{thm:R-uniqueness}] 
Let $\{T_k\}$ be a sequence of times defined by  
	\[
		T_k := \inf\{t\in [S,T]: |\widetilde{Y}(t)-Y(t)|\geq 2^{-k}\}\wedge T, \quad k\in\Nsf.
	\]
Assume that there is no $\varepsilon>0$ such that $Y=\widetilde{Y}$ holds on $[S,S+\varepsilon]$. Then the sequence $\{T_k\}$ converges to $S$ as $k\to\infty$ and it follows (by appealing to superadditivity and continuity of the control function $\omega$) that
	\begin{equation}
	\label{eq:R-uniqueness-convergence}
		\omega(T_{k+1},T_k) \leq \omega(S,T_k) \xrightarrow[k\to\infty]{} 0.
	\end{equation}
Consequently, we can find integers $k$ and $n$ such that
	\begin{equation}
	\label{eq:R-uniqueness-est-1}
		T_{k+1} < T_k < T
			\quad\mbox{and}\quad 
		2^{-(n+1)}\leq \omega(T_{k+1},T_{k}) \leq 2^{-n} < \delta
	\end{equation}
where $\delta$ is as in \autoref{thm:R-uniqueness-L3}. We have 
	\begin{equation}
	\label{eq:R-uniqueness-est-2}
		\sup_{r\in [T_{k+1},T_k]} |\widetilde{Y}(r)-Y(r)| \leq 2^{-k}
	\end{equation}
by the definition of $\{T_k\}$ and we obtain
	\begin{align*}
	2^{-(k+1)} 
		& \leq |\widetilde{Y}(T_{k})-\widetilde{Y}(T_{k+1}) - (Y(T_k)-Y(T_{k+1}))| \\
		& \lesssim |\overline{Z}(T_{k+1},T_k)| 
			+ |\widetilde{Y}(T_{k+1})-Y(T_{k+1})|
			\omega(T_{k+1},T_k)^\frac{1}{p} \\
		& \lesssim \sup_{r\in [T_{k+1},T_k]} |\widetilde{Y}(r)-Y(r)| \omega(T_{k+1},T_k)^\frac{3}{p} + |\widetilde{Y}(T_{k+1})-Y(T_{k+1})| \omega(T_{k+1},T_k)^\frac{1}{p}\\
		& \lesssim 2^{-k-\frac{n}{p}}
	\end{align*}
by using the Reverse Triangle Inequality and the definition of $\{T_k\}$, and then \autoref{thm:R-uniqueness-L1}, \autoref{thm:R-uniqueness-L3}, the upper bound for $\omega$ in \eqref{eq:R-uniqueness-est-1}, and estimate \eqref{eq:R-uniqueness-est-2} successively. It follows that there exists a constant $B\in (0,\infty)$ that only depends on $C$, $\widetilde{C}$, $m$, $\|f\|_3$, $\|X\|$, and $p$ such that $2^{-n} \geq B$. Thus we obtain that $2\omega(T_{k+1},T_k) \geq B$ holds by the lower bound for $\omega$ in \eqref{eq:R-uniqueness-est-1} which contradicts \eqref{eq:R-uniqueness-convergence}.
\end{proof}

\subsection{Existence and uniqueness on manifolds}
\label{sec:M-existence-uniqueness}

In this section, we will prove the existence and uniqueness of integral curves for a rough velocity field on a general manifold $M$. This will be done by gluing integral curves constructed in coordinate patches. We begin with the uniqueness result. 

\begin{theorem}
\label{thm:M-uniqueness}
Let $[S,T]\subseteq\Rsf$ be an interval and let $m,n\in\Nsf$. Let $p\in [1,3)$, let $\omega$ be a control function on interval $[S,T]$, and let $X$ be an $n$-dimensional branched rough path on interval $[S,T]$ with regularity $p$ and control $\omega$. Let $M$ be an $m$-dimensional manifold, let $f: M\to R_n^2M$ be an admissible rough velocity field, and let $Y_1$ and $Y_2$ be two integral curves of $f$ driven by $X$ defined on interval $[S,T]$ such that $Y_1(S)=Y_2(S)$. Then $Y_1=Y_2$ on the whole interval $[S,T]$.
\end{theorem}

\begin{proof}
We prove the claim by contradiction. Suppose that there exists $s\in [S,T]$ such that $Y_1(s)\neq Y_2(s)$ and let
	\begin{equation}
	\label{eq:M-uniqueness:t} 
		t:= \inf\{s\in [S,T]: Y_1(s)\neq Y_2(s)\}.
	\end{equation}
Then $t>S$ and $Y_1=Y_2$ on $[S,t)$ so that $Y_1(t)=Y_2(t)$ holds as well by continuity. If $t=T$, the proof is finished. Let us therefore assume that $t<T$. Let $Y_0:=Y_1(t)=Y_2(t)$ and choose a chart $(U,x)$ around $Y_0$. Let $V$ be a compact neighborhood of $Y_0$ contained in $U$, let $K$ be a compact neighborhood of $Y_0$ contained in $V$, and let $\delta: M\to [0,1]$ be a smooth bump function supported in $V$ with unit value on $K$. Let $\widetilde{f}: U\to R_n^2 U$ be the restriction of the rescaling of $f$ by $\delta$ to $U$ and let $(\widetilde{f}^i_{\tau})$ be the coordinate representation of $\widetilde{f}$ in the induced chart $(\overline{U},\overline{x})$ on $R_n^2M$. Let $g: \Rsf^m\to R_n^2\Rsf^m$ be the rough velocity field that is obtained by extending the functions $\widetilde{f}^i_{\tau} \circ {x}^{-1}: x(U)\to \Rsf$ (for $i\in\llbracket 1,m \rrbracket$, $\tau\in\Fcal_n^2$) to the whole of $\Rsf^m$ by zero outside of $x(U)$. (Such extension is smooth as $\delta$ vanishes near the boundary of $U$.) It follows that $g$ is compactly supported, admissible, and $x$-related to $\widetilde{f}$.

By continuity of $Y_1$ and $Y_2$, there exists $\theta\in (0,\infty)$ such that both $Y_1$ and $Y_2$ take values in $K$ on the interval $[t,t+\theta]$. Denote by $\widetilde{Y}_1$ and $\widetilde{Y}_2$ the restrictions of $Y_1$ and $Y_2$ to $[t,t+\theta]$, respectively. Since $f=\widetilde{f}$ on $K$, it follows that both $\widetilde{Y}_1$ and $\widetilde{Y}_2$ are integral curves of $\widetilde{f}$ driven by $X$. Hence, both $x\circ\widetilde{Y}_1$ and $x\circ\widetilde{Y}_2$ are integral curves of $g$ driven by $X$ by \autoref{thm:related-fields} such that both start at $x(Y_0)$. But then there exists $\varepsilon\in (0,\infty)$ such that $x\circ\widetilde{Y}_1 = x\circ\widetilde{Y}_2$ on $[t,t+\varepsilon]$ by \autoref{thm:R-uniqueness}. But then also $Y_1=Y_2$ on $[t,t+\varepsilon]$ which contradicts the definition of $t$ in~\eqref{eq:M-uniqueness:t}.
\end{proof}

We continue with two basic existence theorems. First, we give a lemma that transfers the global existence result from \autoref{thm:R-existence} to a manifold.

\begin{lemma}
\label{thm:M-local-existence-1}
Let $[S,T]\subseteq\Rsf$ be an interval and let $n\in\Nsf$. Let $p\in [1,3)$, let $\omega$ be a control function on interval $[S,T]$, and let $X$ be an $n$-dimensional branched rough path on interval $[S,T]$ with regularity $p$ and control $\omega$. Let $M$ be a manifold, let $f: M\to R_n^2M$ be an admissible rough velocity field, and let $q\in M$. Then there exists $\varepsilon\in (0,\infty)$ and a neighborhood $U$ of $q$ such that for every $Y_0\in U$ and every $S_0\in [S,T]$, there exists an integral curve $Y: [S_0,(S_0+\varepsilon)\wedge T]\to M$ of $f$ driven by $X$ that satisfies $Y(S_0)=Y_0$. 
\end{lemma}

\begin{proof}
Let $(V,y)$ be a chart around $q$, let $W$ be a compact neighborhood of $q$ contained in $V$, let $K$ be another compact neighborhood of $q$ contained in $W$, and let $\delta:M\to [0,1]$ be a smooth bump function supported in $W$ that takes unit value on $K$. Let $\widetilde{f}:V\to R_n^2V$ be the restriction of the rescaling of $f$ by $\delta$ to $V$. As in the proof of \autoref{thm:M-uniqueness}, let $g: \Rsf^m \to R_n^2\Rsf^m$ be the rough velocity field obtained by extending the representation of $\widetilde{f}$ in the induced chart $(\overline{V},\overline{y})$ on $R_n^2M$ expressed in coordinates from $y(V)$ to the whole of $\Rsf^m$ by zero outside of $y(V)$. 

Then $g$ is compactly supported and admissible so that, by \autoref{thm:R-existence}, for every $c_0\in\Rsf^m$ and every $S_0\in [S,T]$, there is a function $c: [S_0,T]\to \Rsf^m$ satisfying $c(S_0)=c_0$ and such that the inequality
	\begin{equation}
	\label{eq:M-existence-lemma-1}
		\left|c^i(t) - c^i(s) - \sum_{\tau\in\Fcal_n^2} \frac{1}{\sigma(\tau)} g^i_\tau(c(s)) X^\tau(s,t)\right| \leq C\omega(s,t)^\frac{3}{p}
	\end{equation}
holds for every $i\in\llbracket 1,m \rrbracket$ and all $(s,t)\in \Delta[S_0,T]$. (The constant in the above estimate can be chosen independently of $S_0$ and $c_0$ as it only depends on $m$, $n$, $\|g\|_3$, and $\|X\|$.) It follows that 
	\begin{equation}
	\label{eq:M-existence-lemma-2}
		|c^i(t)-c^i(s)| \lesssim \omega(s,t)^\frac{3}{p} + \omega(s,t)^\frac{2}{p} + \omega(s,t)^\frac{1}{p} \leq \widetilde{C}\omega(s,t)^\frac{1}{p}
	\end{equation}
holds for every $i\in\llbracket 1,m\rrbracket$ and $(s,t)\in\Delta[S_0,T]$ with $\widetilde{C}\in (0,\infty)$ that only depends on $m$, $n$, $\|f\|_3$, and $\|X\|$.

Now, let $r\in (0,\infty)$ be such that the ball with center $y(q)$ and radius $r$, $B(y(q),r)$, is contained in $y(K)$. Let also $U:= y^{-1}(B(y(q),r/2))$ and let $\varepsilon\in (0,\infty)$ be such that $\widetilde{C}\omega(s,t)^{1/p} \leq r/2$ holds for all $(s,t)\in\Delta[S,T]$ satisfying $|s-t|\leq \varepsilon$. 
Let $Y_0\in U$ and $S_0\in[S,T]$, and let $c:[S_0,T]\to\Rsf^m$ be the function such that $c(S_0)=y(Y_0)$ and such that~\eqref{eq:M-existence-lemma-1} holds. Then
	\[ 
		|c^i(t) - y^i(q)| 
			\leq |c^i(t)-c^i(S_0)| + |y^i(Y_0)-y^i(q)| 
			\leq \widetilde{C}{\omega}(t,S_0)^\frac{1}{p} + \frac{r}{2} 
			\leq r
	\]
holds for all $t\in [S_0,(S_0+\varepsilon)\wedge T]$ by appealing to~\eqref{eq:M-existence-lemma-2} so that $c$ stays in $B(y(q),r)\subseteq y(K)$ up to time $(S_0+\varepsilon)\wedge T$. Now, as $\widetilde{f}$ is $y^{-1}$-related to $g$ and as $c: [S_0, (S_0+\varepsilon)\wedge T] \to y(K)$ is an integral curve of $g$ driven by $X$, we have by \autoref{thm:related-fields} that the function $Y: [S_0, (S_0+\varepsilon)\wedge T] \to K$ defined by $Y:=y^{-1}\circ c$ is an integral curve of $\widetilde{f}$ driven by $X$ that satisfies $Y(S_0)=Y_0$. Since $\widetilde{f}=f$ on $K$, $Y$ is also an integral curve of $f$ driven by $X$.
\end{proof}

The local existence theorem is a special case of the above lemma. 

\begin{theorem}
\label{thm:M-local-existence}
Let $[S,T]\subseteq\Rsf$ be an interval and let $n\in\Nsf$. Let $p\in [1,3)$, let $\omega$ be a control function on interval $[S,T]$, and let $X$ be an $n$-dimensional branched rough path on interval $[S,T]$ with regularity $p$ and control $\omega$. Let $M$ be a manifold, let $f: M\to R_n^2M$ be an admissible rough velocity field, and let $Y_0\in M$. Then there exists $\varepsilon\in (0,\infty)$ and an integral curve $Y:[S,(S+\varepsilon)\wedge T]\to M$ of $f$ driven by $X$ that satisfies $Y(S)=Y_0$. 
\end{theorem}

As the next example shows, one cannot expect global existence without additional assumptions.

\begin{example}
\label{ex:local-but-not-global-solution}
Let $X=(X^\tau)$ be the (non-geometric) $2$-dimensional branched rough path (with regularity $p=2$ and control function $\omega(s,t)=t-s$) defined by 
	\[
		X^{\tau}(s,t) := \begin{cases}
							t-s, & \quad \tau = \tcjd, \\
							0, & \quad \mbox{otherwise},
						\end{cases}
	\]
for $(s,t)\in \Delta[0,1]$. Let $f:\Rsf\to R_2^2 \Rsf$ be the (admissible but not compactly supported) rough velocity field defined by 
	\[
		f_{\tau}(y) := \begin{cases}
							y^2, & \quad \tau = \tcjd,\\
							-y^2, & \quad \tau = \tjd,\\
							0, & \quad \mbox{otherwise},
						\end{cases}
	\]
for $y\in \Rsf$. Then, by \autoref{thm:M-local-existence}, for every $Y_0 >0$, there exists $\varepsilon\in (0,\infty)$ and an integral curve $Y: [0, \varepsilon\wedge 1]\to \Rsf$ of $f$ driven by $X$ that satisfies $Y(0)=Y_0$. That is, we have that 
	\[ 
		Y(t) - Y(s) = Y(s)^2 (t-s) + r(s,t)	\]
holds for all $(s,t)\in \Delta [0,\varepsilon\wedge 1]$ where the remainder satisfies $|r(s,t)| \leq C(t-s)^\frac{3}{2}$ for a suitable constant $C\in (0,\infty)$ and every $(s,t)\in \Delta [0,\varepsilon\wedge 1]$, i.e.\ $Y$ is the solution to the ODE $\dot{Y}=Y^2$ on $[0,\varepsilon \wedge 1]$. It is well-known that this equation has a finite blow-up time $1/Y_0$.
\end{example}

In what follows, we prove that if the rough velocity field is compactly supported, then a global solution (i.e.\ one that is defined on the maximal possible interval) exists. To this end, we first show that if the rough velocity field is compactly supported, then we may take $\varepsilon$ in \autoref{thm:M-local-existence-1} independent of the initial condition. 

\begin{lemma}
\label{thm:M-local-existence-2}
Let $[S,T]\subseteq\Rsf$ be an interval, and let $n\in\Nsf$. Let $p\in [1,3)$, let $\omega$ be a control function on interval $[S,T]$, and let $X$ be an $n$-dimensional branched rough path on interval $[S,T]$ with regularity $p$ and control $\omega$. Let $M$ be a manifold and let $f:M\to R_n^2M$ be a compactly supported admissible rough velocity field. Then there exists $\varepsilon\in (0,\infty)$ such that for every $Y_0\in M$ and every $S_0\in [S,T]$, there exists an integral curve $Y: [S_0, (S_0+\varepsilon)\wedge t]\to M$ of $f$ driven by $X$ that satisfies $Y(S_0)=Y_0$.
\end{lemma}

\begin{proof}
Let $K$ be the support of $f$. Then for each $q\in K$, there exists a neighborhood $U_q$ of $q$ and $\varepsilon_q\in (0,\infty)$ such that for each $Y_0\in U_q$ and each $S_0\in [S,T]$, there exists an integral curve of $f$ starting from $Y_0$ defined on $[S_0,(S_0+\varepsilon_q)\wedge T]$ by \autoref{thm:M-local-existence-1}. Since $\{U_q\}_{q\in K}$ is a cover of $K$, we can find finitely many points $q_1, q_2, \ldots, q_N$ ($N\in \Nsf$) such that $\{U_{q_k}\}_{k\in \llbracket 1,N \rrbracket}$ still cover $K$. Set $\varepsilon:= \min_{k\in \llbracket 1,N \rrbracket} \varepsilon_{q_k}$. Then for each $Y_0\in K$ and each $S_0\in [S,T]$, there is an integral curve $Y: [S_0, (S_0+\varepsilon)\wedge T]\to M$ of $f$ driven by $X$ starting from $Y_0$. For each $Y_0\not\in K$ and each $S_0\in [S,T]$, we simply set $Y:=Y_0$ to obtain an integral curve with the desired property.
\end{proof}

Now we may glue together the local integral curves obtained by the above lemma to produce a global integral curve. Because each of the local integral curves is guaranteed to be at least $\varepsilon$ long, it is possible to cover the whole interval $[S,T]$ in finitely many steps.

\begin{theorem}
\label{thm:M-global-existence}
Let $[S,T]\subseteq\Rsf$ be an interval, and let $n\in\Nsf$. Let $p\in [1,3)$, let $\omega$ be a control function on interval $[S,T]$, and let $X$ be an $n$-dimensional branched rough path on interval $[S,T]$ with regularity $p$ and control $\omega$. Let $M$ be a manifold and let $f:M\to R_n^2M$ be a compactly supported admissible rough velocity field. Then for each $Y_0\in M$, there exists an integral curve $Y: [S,T]\to M$ of $f$ driven by $X$ such that $Y(S)=Y_0$. 
\end{theorem}

\begin{proof}
Let $\varepsilon$ be as in \autoref{thm:M-local-existence-2}. Let us define a sequence of points $\{q_k\}$, curves $\{y_k\}$, and times $\{S_k\}$ recursively as follows: Let $S_0:=S$ and $q_0:=Y_0$. Given $q_k\in M$ and $S_k\in [S,T]$, let $y_k: [S_k,(S_k+\varepsilon)\wedge T]\to M$ be an integral curve of $f$ driven by $X$ such that $y_k(S_k)=q_k$, let $S_{k+1}:=(S_k+\varepsilon/2)\wedge T$, and let $q_{k+1}:=y_k(S_{k+1})$. Let $N$ be the smallest integer such that $S_N=T$. Finally, define $Y:[S,T]\to M$ to be equal to $y_k$ on $[S_k,S_{k+1}]$ for $k\in \llbracket 0,N-1 \rrbracket$. 

We now show that $Y$ is an integral curve of $f$ driven by $X$. Note first that $Y$ is clearly continuous. Now, let $r\in [S,T]$.

If $Y$ is equal to $y_k$ in a neighborhood $D$ of $r$, then, because $y_k$ is an integral curve, we can take any chart $(U,x)$ around $Y(r)$ and any compact interval $I\subseteq D\cap Y^{-1}(U)$ containing $Y(r)$ in its interior and we have that there exists $C\in (0,\infty)$ such that the inequality
	\[
		\left|
			Y^i(t) - Y^i(s) 
			- \sum_{\tau\in\Fcal_n^2}\frac{1}{\sigma(\tau)} f^i_\tau(Y(s)) X^\tau(s,t)
		\right| 
		\leq 
		C\omega(s,t)^\frac{3}{p},
	\]
where $(Y^i)$ is the coordinate representation of $Y$ in chart $(U,x)$ and $(f^i_\tau)$ is the coordinate representation of $f$ in the induced chart $(\overline{U},\overline{x})$ on $R_n^2M$, holds for every $i\in\llbracket 1,m \rrbracket$ and $(s,t)\in \Delta I$.

If $r$ is equal to $S$ or $T$, or if $r\in(S_k, S_{k+1})$ for some $k\in\llbracket 0,N-1 \rrbracket$, then the neighborhood $D$ of $r$ on which $Y=y_k$ holds clearly exists by construction of $Y$. If $r=S_k$ for some $k\in \llbracket 1,N-1 \rrbracket$, then by its construction, $Y$ is equal to $y_{k-1}$ on $[S_{k-1},r]$ and equal to $y_k$ on $[r, S_{k+1}]$. But $y_{k-1}$ is defined on the whole interval $[S_{k-1}, S_{k+1}]$ so that its restriction to $[r,S_{k+1}]$ is an integral curve of $f$ driven by $X$ starting at $y_{k-1}(r)$. On the other hand, the restriction of $y_{k}$ to interval $ [r, S_{k+1}]$ is also an integral curve of $f$ driven by $X$ starting at $y_{k-1}(r)$ so that by \autoref{thm:M-uniqueness}, we have that $y_{k}=y_{k-1}$ on $[r, S_{k+1}]$. Hence, $Y=y_{k-1}$ holds on $[S_{k-1}, S_{k+1}]$ and, in particular, on a neighborhood $D$ of $r$.
\end{proof}

\section{Invariance of submanifolds for RDEs}
\label{sec:invariance}

In what follows, we give a characterization of invariant submanifolds for solutions to RDEs. Recall that if $M$ and $M'$ are two manifolds such that $M$ is properly embedded in $M'$ and if $\phi:M\to M'$ is the embedding, then by \autoref{thm:r2-preserve-embedding}, the mapping $R_n^2\phi: R_n^2M\to R_n^2M'$ defined by $[(j_0^2\psi,f)]\mapsto [(j_{\psi(0)}^2 (\phi\circ \psi),f)]$, is also an embedding. 

Roughly speaking, the following theorem says that if $f$ is a rough velocity field on manifold $M'$, then any submanifold $M\subseteq M'$ on which $f$ takes values in $R_n^2M\subseteq R_n^2M'$ is invariant for an RDE corresponding to $f$, i.e.\ if a solution to an RDE corresponding to $f$ starts in $M$, it stays in $M$. 

\begin{theorem}
\label{thm:invariance}
Let $n\in\Nsf$, let $M'$ be a manifold and let $f: M'\to R_n^2M'$ be a rough velocity field. Let also $M$ be a properly embedded submanifold of $M'$ with $\phi: M\to M'$ being the embedding. Then the following statements are equivalent:
	\begin{enumerate}[label=(\arabic*)]
	\itemsep0em
	\item\label{invariance-condition-1} It holds that $(f\circ\phi) (M)\subseteq \mathrm{Im}( R_n^2\phi )$.
	\item\label{invariance-condition-2} For every interval $[S,T]\subseteq\Rsf$, every $n$-dimensional branched rough path $X$ on interval $[S,T]$ with regularity $p\in [1,3)$ and control function $\omega$, and every integral curve $Y:[S,T]\to M'$ of $f$ driven by $X$ such that $Y(S)\in \phi(M)$, we have that $Y(t)\in \phi(M) $ for every $t\in [S,T]$.
	\end{enumerate}
\end{theorem}

\begin{proof}
Let us prove the implication \ref{invariance-condition-1} $\implies$ \ref{invariance-condition-2} first. Suppose that \ref{invariance-condition-1} holds and let $[S,T]\subseteq\Rsf$ be an interval, $X$ be an $n$-dimensional branched rough path on interval $[S,T]$ with regularity $p\in [1,3)$ and control function $\omega$, and let $Y:[S,T]\to M'$ be an integral curve of $f$ driven by $X$ such that $Y(S)\in \phi(M)$. Suppose to the contrary of \ref{invariance-condition-2} that there exists $t\in (S,T]$ such that $Y(t)\not\in \phi(M)$ and let $s:=\inf\{u\in [S,T]: Y(u)\not\in \phi(M)\}$. We note that $Y(s)\in \phi(M)$. (This is because if $s=S$, then $Y(s)=Y(S)\in \phi(M)$ holds by assumption and if $s>S$, then $Y(s)\in \phi(M)$ holds since $Y$ is continuous and $\phi(M)$ is closed.) By \ref{invariance-condition-1}, the map $\phi\circ f$ can be understood to take values in $R_n^2M$ and so it gives rise to a rough velocity field $g:M\to R_n^2M$. This field is clearly admissible. It follows by \autoref{thm:M-local-existence-2}, that there exists $\varepsilon\in (0,\infty)$ and an integral curve $Z: [s,(s+\varepsilon)\wedge T]\to M$ of $g$ driven by $X$ satisfying $\phi(Z(s)) = Y(s)$. As $f$ is clearly $\phi$-related to $g$ (i.e.\ $f(\phi(p)) = (R_n^2\phi)(g(p))$ holds for every $p\in M$), it follows, by \autoref{thm:related-fields}, that the function $\phi\circ Z: [s,(s+\varepsilon)\wedge T]\to \phi(M) \subseteq M'$ is an integral curve of $f$ driven by $X$ starting from $Y(s)$. This means, by \autoref{thm:M-uniqueness}, that $Y=\phi\circ Z$ holds on $[s,(s+\varepsilon)\wedge T]$ but this cannot happen since $Y$ is not $\phi(M)$-valued on $(s,(s+\varepsilon)\wedge T]$ by the definition of $s$ while $\phi\circ Z$ is $\phi(M)$-valued on $[s,(s+\varepsilon)\wedge T]$ by its construction.

Let us prove the implication \ref{invariance-condition-2} $\implies$ \ref{invariance-condition-1} now. Recall that given a point $q\in M$, by the Rank Theorem, we may choose charts $(V,y)$ around $q$ and $(U,x)$ around $\phi(q)$ in which $\phi$ has the representation 
	\[ 
		\phi(x^1, x^2, \ldots, x^m) = (x^1, x^2, \ldots, x^m, \underbrace{0,0, \ldots, 0}_{(m'-m)-\times})
	\]
where $m$ and $m'$ are the dimensions of $M$ and $M'$, respectively. Let $(f^i_\tau)$ be the coordinate representation of $f$ in $(U,x)$. Suppose that for some $q_0\in V$, $i_0\in\llbracket m+1,m' \rrbracket$, and $\tau_0\in \Fcal_n^2$, we have $f^{i_0}_{\tau_0}(\phi(q_0))\neq 0$.

\emph{Step 1.} First, suppose that $\tau_0=\tazero$ for some $\alpha_0\in \llbracket 1,n \rrbracket$. Let $X = (X^\tau)$ be the branched rough path canonically built from the smooth path $x^{\alpha_0}(t) = t$ and $x^{\alpha}(t) = 0$ for $\alpha\neq \alpha_0$ and $t\in [0,1]$, i.e.\ we have
	\[
		X^\tau(s,t) := \begin{cases}
							t-s, & \quad \tau =\tazero, \\
							\frac{1}{2}(t-s)^2, & \quad \tau = \tvazero, \\
							(t-s)^2, & \quad \tau = \taazero,\\
							0, & \quad \mbox{otherwise},
						\end{cases}
	\]
for $(s,t)\in\Delta[0,1]$. (Such branched rough path $X$ has regularity $p=1$ and control function $\omega(s,t):= t-s$ for $(s,t)\in \Delta[0,1]$.) By \autoref{thm:M-local-existence}, there exists $\varepsilon\in (0,1)$ and an integral curve $Y: [0,\varepsilon] \to M'$ of $f$ driven by $X$ such that $Y(0)=\phi(q_0)$. Then, by using admissibility of $f$ and the fact that $Y^{i_0}(\phi(q_0))=0$, we have the existence of a constant $C\in (0,\infty)$ such that the inequality
\[
	\left|
		Y^{i_0}(t) - f^{i_0}_{\tazero}(\phi(q_0)) t 
		- \frac{1}{2} \left(\frac{\partial f^{i_0}_{\tazero}}{\partial x^j} f^j_{\tazero}\right)(\phi(q_0))t^2
	\right| 
	\leq C t^3
\]
holds for all sufficiently small $t$. However, since $f^{i_0}_{\tazero}(\phi(q_0))\neq 0$, this not possible unless $Y^{i_0}(t)\neq 0$ for such small $t$, i.e.\ unless $Y$ leaves $M$. But since this cannot happen by \ref{invariance-condition-2}, we must have that $f^i_{\ta} = 0$ on $U\cap\phi(V)$ for all $i\in \llbracket m+1, m' \rrbracket$ and $\alpha\in \llbracket 1,n \rrbracket$.

\emph{Step 2.} Second, suppose that $\tau=\tczero$ for some $\alpha_0,\beta_0\in\llbracket 1,n \rrbracket$. Let $X=(X^\tau)$ be the pure-area rough path on interval $[0,1]$, i.e.\ the branched rough path defined by
	\[
		X^\tau(s,t) = \begin{cases}
						t-s, & \quad \tau=\tczero,\\
						0, & \quad\mbox{otherwise},
					  \end{cases}
	\]
for $(s,t)\in\Delta[0,1]$. (Such branched rough path $X$ has regularity $p=2$ and control function $\omega(s,t) = t-s$ for $(s,t)\in \Delta[0,1]$.) Then, similarly as in \emph{Step 1}, by \autoref{thm:M-local-existence}, there exists $\varepsilon\in (0,1)$ and an integral curve $Y: [0,\varepsilon] \to M'$ of $f$ driven by $X$ such that $Y(0)=\phi(q_0)$. For this curve, there exists a constant $C\in (0,\infty)$ such that the inequality 
	\[
		\left|
			Y^{i_0}(t) - f^{i_0}_{\tczero}(\phi(q_0))t
		\right|
		\leq C t^\frac{3}{2}
	\]
holds for all sufficiently small $t$. This, again, is impossible unless $Y^{i_0}(t)\neq 0$ for such small $t$. But since this cannot happen by \ref{invariance-condition-2}, we must have that $f^i_{\tc}=0$ on $U\cap \phi(V)$ for all $i\in \llbracket m+1, m' \rrbracket$ and $\alpha, \beta\in \llbracket 1,n \rrbracket$. 

\emph{Step 3.} Finally, for $i\in \llbracket m+1, m' \rrbracket$, $\alpha,\beta\in\llbracket 1,n \rrbracket$, and $p\in U\cap \phi(V)$ we have 
	\[
		f^i_{\tab}(p) 
			= \frac{\partial f^i_{\tb}}{\partial x^j}(p)f^j_{\ta}(p)
				- f^i_{\tc}(p)
	\]
by admissibility of the rough velocity field $f$ and since $f^i_{\tc}(p)=0$ and $f^i_{\tb}(p)=0$ (and, therefore, also $\frac{\partial f^i_{\tb}}{\partial x^j}(p) = 0$ for $j\in \llbracket 1,m\rrbracket$) hold by \emph{Steps 1} and \emph{2}, we see that $f^i_{\tab}(p)=0$.
\end{proof}

In what follows, we provide a description of bundle $R_n^2M$ for a manifold $M$ specified as a zero-set of a function. Recall that zero-sets of smooth submersions are properly embedded; see, e.g., \cite[Corollary 5.13]{Lee12}. 

\begin{theorem}
\label{thm:zero-set}
Let $F: \Rsf^m\to \Rsf^{m'}$ be a smooth submersion and let $M$ be its zero-set, i.e.\ 
	\[ 
		M := \{ x\in\Rsf^m : F(x)=0 \}.
	\]
Then for $n\in\Nsf$, we have
	\[ 
		R_n^2M = \{ \bm{f}\in R_n^2\Rsf^m: (R_n^2 F)(\bm{f}) = 0 \}.
	\]
\end{theorem}

\begin{proof}
Denote 
	\[ 
		A:= \{\bm{f}\in R_n^2\Rsf^m: (R_n^2F)(\bm{f}) = 0\}
	\]
and let $\phi: M\to\Rsf^m$ be the embedding of $M$ into $\Rsf^m$. 

First note that the inclusion $R_n^2 M\subseteq A$ follows from the equality
	\[ 
		((R_n^2F) \circ (R_n^2\phi))(\bm{f}) = (R_n^2 (F\circ \phi))(\bm{f}) = 0
	\]
that holds for all $\bm{f}\in R_n^2M$ by using \autoref{thm:r2-functorial} and the fact that we have  $F\circ\phi =0$ by definition. 

To prove that the inclusion $A\subseteq R_n^2M$ also holds, let $\bm{f}\in A$ and denote its coordinates in the global chart by $(y^i,f^i_\tau)$. We have that $F(\pi(\bm{f})) = \pi ((R_n^2 F)(\bm{f}))=0$ so that $\pi(\bm{f})=\phi(p)$ for some $p\in M$. (Recall that $\pi$ is the bundle projection map.) Choose a chart $(U,x)$ around $p$ and let 
	\[
		a_j^i := \frac{\partial\phi^i}{\partial x^j}(p), 
			\quad 
		a^i_{jk} := \frac{\partial^2 \phi^i}{\partial x^j\partial x^k}(p), 
			\quad (i,j,k\in \llbracket 1,m \rrbracket)
	\]
and 
	\[
		b_k^j := \frac{\partial F^j}{\partial y^k}(\phi(p)), 
			\quad 
		b_{k\ell}^j := \frac{\partial^2 F^j}{\partial y^k\partial y^\ell}(\phi(p))
			\quad (j\in \llbracket 1,m' \rrbracket, \, k,\ell\in \llbracket 1,m \rrbracket).
	\]
We aim to find $(g^j_\tau)$ such that $R_n^2\phi$ maps the point in $R_n^2M$ with coordinates $(x^j(p), g^j_\tau)$ to $\bm{f}$, i.e.\ we aim to find $(g^j_\tau)$ such that 
	\begin{equation}
	\label{eq:g-cond}
		f^i_{\ta} = a^i_jg^j_{\ta}, 
			\quad 
		f^i_{\tc} = a^i_j g^j_{\tc}, 
			\quad \mbox{and}\quad 
		f^i_{\tab} = a^i_j g^j_{\tab} + a^i_{jk}g^j_{\ta} g^k_{\tb}
	\end{equation}
hold for all $i\in\llbracket 1,m \rrbracket$ and $\alpha,\beta\in\llbracket 1,n \rrbracket$. As $\bm{f}\in A$, we have that 
	\begin{equation}
	\label{eq:f-assumption}
		b_k^j f^k_{\ta}=0, 
			\quad 
		b^j_kf^k_{\tc} = 0, 
			\quad \mbox{and}\quad 
		b^j_k f^k_{\tab} + b^j_{rs}f^r_{\ta} f^s_{\tb} = 0
	\end{equation}
hold for every $j\in\llbracket 1,m' \rrbracket$ and every $\alpha, \beta\in\llbracket 1,n \rrbracket$. Let $a$ be the matrix with components $(a^i_j)$ and $b$ the matrix with components $(b^j_k)$. Then because $\mathrm{Im}\, a\subseteq\mathrm{Ker}\, b$ and because these spaces have the same dimension $(m-m')$, we have the following: For every $(f^i)\in\Rsf^m$ such that $b^j_k f^k=0$ holds for every $j\in \llbracket 1,m' \rrbracket$, there exists $(g^j)\in \Rsf^{m-m'}$ such that $a_j^ig^j = f^i$ for every $i\in \llbracket 1,m \rrbracket$. It follows from this fact and the first two equations in~\eqref{eq:f-assumption}, that $(g^j_{\ta})$ and $(g^j_{\tc})$ satisfying the first two equations in~\eqref{eq:g-cond} exist. Now, by the chain rule, we have that
	\[ 
		 b^j_r a^r_{k\ell} + b^j_{pq}a_k^p a^q_{\ell} = 0
	\]
holds for every $j\in\llbracket 1,m' \rrbracket$ and $k,\ell\in \llbracket 1,m \rrbracket$, so that 
	\[
    	b^{j}_{k}\left(
    				f^{k}_{\tab} - a^{k}_{pq}g^{p}_{\ta}g^{q}_{\tb}
    			  \right)
    	  = b^{j}_{k}f^{k}_{\tab} - b^{j}_{k} a^k_{pq} g^{p}_{\ta} g^q_{\tb} 
  		  = b^{j}_{k}f^{k}_{\tab} + b^{j}_{rs}a^{p}_{r}a^{q}_{s} g^{p}_{\ta}g^{q}_{\tb}
  		  = b^{j}_{k}f^{k}_{\tab} + b^{j}_{rs}f^{r}_{\ta}f^{s}_{\tb} = 0
  	\]
holds for every $j\in\llbracket 1,m' \rrbracket$ by the third equation in~\eqref{eq:f-assumption}. Thus we can find $(g^j_{\tab})$ such that 
	\[
		a_j^ig^j_{\tab} = f^i_{\tab} - a^i_{pq}g^p_{\ta} g^q_{\tb}
	\]
holds for every $i\in\llbracket 1,m \rrbracket$ and every $\alpha,\beta\in \llbracket 1,n \rrbracket$ which implies the third equation in~\eqref{eq:g-cond}.
\end{proof}

\begin{example} 
Consider the vector field $f: (y^1,y^2)\mapsto (-y^2,y^1)$ that is tangent to circles of the form $M:=\{y\in\Rsf^2: F(y)=0\}$ for $F(y^1,y^2):= (y^1)^2 + (y^2)^2 - r^2$ where $r \in (0,\infty)$. The rough velocity field $(f^i_\tau)$ associated with $f$ via equations~\eqref{eq:canonical-rvf} is given by 
	\[
	\begin{array}{lcl}
		  f^1_{\tao}(y^1,y^2) = - y^2, & & f^2_{\tao} (y^1,y^2) = y^1, \\
		  f^1_{\taboo}(y^1,y^2) = 0,   & & f^2_{\taboo}(y^1,y^2) = 0,\\
		  f^1_{\tcoo}(y^1,y^2) = -y^1, & & f^2_{\tcoo}(y^1,y^2) = - y^2
	\end{array}
	\]
and it is clearly admissible. Let $g:= (R_1^2F)(f)$ and compute its components:
	\begin{align*}
			g_{\tao}(y^1,y^2) 
				& = \left(
						\frac{\partial F}{\partial y^1} f^1_{\tao}
					\right)(y^1,y^2) 
					+ \left(
						\frac{\partial F}{\partial y^2} f^2_{\tao}
					  \right)(y^1,y^2) 
				  = -2y^1y^2 + 2y^2y^1 
				  = 0, \\
			g_{\taboo}(y^1,y^2) 
				& = \left(
						\frac{\partial F}{\partial y^j}f^j_{\taboo}
					\right)(y^1,y^2) 
					+ \left(
						\frac{\partial^2 F}{\partial y^jy^k}f^j_{\tao}f^k_{\tao}
					  \right)(y^1,y^2) 
				  = 2(y^2)^2 + 2(y^1)^2,\\
			g_{\tcoo}(y^1,y^2) 
				& = \left(
						\frac{\partial F}{\partial y^1} f^1_{\tcoo}
					\right)(y^1,y^2) 
					+ \left(
						\frac{\partial F}{\partial y^2} f^2_{\tcoo}
					  \right)(y^1,y^2) 
				  = -2(y^1)^2 - 2(y^2)^2.
	\end{align*}
We see that not all the components of $g$ are zero and therefore, by \autoref{thm:zero-set}, the rough velocity field $(f^i_{\tau})$ when considered on $M$ does not take values in $R_1^2M$. However, if we consider the rough velocity field $(\widetilde{f}^i_\tau)$ given by 
	\[
	\begin{array}{lcl}
		  \widetilde{f}^1_{\tao}(y^1,y^2) := - y^2, & & \widetilde{f}^2_{\tao} (y^1,y^2) := y^1, \\
		  \widetilde{f}^1_{\taboo}(y^1,y^2) := -y^1,   & & \widetilde{f}^2_{\taboo}(y^1,y^2) := -y^2,\\
		  \widetilde{f}^1_{\tcoo}(y^1,y^2) := 0, & & \widetilde{f}^2_{\tcoo}(y^1,y^2) := 0, 
	\end{array}
	\]
then by performing the same calculations as above, it can be seen that this rough velocity field when considered on $M$ does take values in $R_1^2M$. It follows by \autoref{thm:invariance} that $M$ will not be invariant for any integral curve of $(f^i_\tau)$ but it will be invariant for any integral curve of $(\widetilde{f}^i_\tau)$.

To put these results into perspective, we note that if $X=(X^\tau)$ is a scalar branched rough path, then the integral curve $Y$ of $(f^i_\tau)$ driven by $X$ can be locally approximated as
	\[ 
		Y(t) \approx Y(s) + f(Y(s)) X^{\tao}(s,t) - Y(s)X^{\tcoo}(s,t) 
	\]
and the integral curve $\widetilde{Y}$ of $(\widetilde{f}^i_\tau)$ driven by $X$ can be locally approximated as 
	\[
		\widetilde{Y}(t) \approx \widetilde{Y}(s) + f(\widetilde{Y}(s))X^{\tao}(s,t)  - \frac{1}{2} \widetilde{Y}(s) [X](s,t) - \widetilde{Y}(s) X^{\tcoo}(s,t)
	\]
where $[X]:= X^{\taboo} - 2X^{\tcoo}$ (i.e.\ $[X]$ denotes the \emph{bracket} of the rough path $X$). In the particular case when $X$ is the It\^o rough path lift of a continuous local martingale $x$, these approximations mean that $Y$ is the (rough-path) solution to the It\^o SDE
	\[ 
		\d{Y} = f(Y)\d{x}
	\]
while $\widetilde{Y}$ solves the It\^o SDE
	\[
		\d{\widetilde{Y}} = f(\widetilde{Y})\d{x} - \frac{1}{2} \widetilde{Y}\d\langle x\rangle
	\]
(in this case, $[X]$ coincides with the quadratic variation $\langle x\rangle$ of $x$). Using the standard flat connection, the above It\^o SDE can be rewritten to the Stratonovich SDE 
	\[ 
		\d\widetilde{Y} = f(\widetilde{Y})\circ \d{x}
	\]
and we can immediately apply the classical criteria (see, e.g., \cite{AubDaP90,Mil97}) to see that our results are consistent with the known results for SDEs.
\end{example}

\section{Concluding remark}

In the article, we developed a novel framework based on natural bundles in which RDEs on manifolds driven by branched rough paths of regularity $p\in [1,3)$ can be treated. We believe that the extension of our results to the case of rough paths of arbitrarily low regularity $p\in [1,\infty)$ is possible but it would be technically demanding. In order to do so, one would first replace $\mathcal{F}_n^2$ by the space $\mathcal{F}_n^k$ of forests of rooted trees with at most $k:=\lfloor p\rfloor$ vertices. The notions of rough velocities, jet-rough velocity composition, rough velocity bundles, and rough velocity fields as well as that of integral curves would then be extended to this case analogously as in \autoref{sec:rvf} and \autoref{sec:RDEs}, respectively. A higher-order admissibility condition would also appear but it seems that a nontrivial abstract algebraic framework would have to be developed to characterize this condition in full generality. Additionally, the rescaling operation, compatible with the higher-order admissibility condition, would need to be introduced to allow for gluing solutions to RDEs constructed in coordinate patches.

\end{document}